\documentclass[reqno]{amsart}
\usepackage{amsmath, amsthm, amssymb, amstext, mathrsfs}

\usepackage[left=3cm,right=3cm,top=3cm,bottom=3cm]{geometry}
\usepackage{hyperref,xcolor}
\hypersetup{
 pdfborder={0 0 0},
 colorlinks,
}
\usepackage{enumitem}
\allowdisplaybreaks

\newtheorem{theorem}{Theorem}
\newtheorem{remark}[theorem]{Remark}
\newtheorem{lemma}[theorem]{Lemma}
\newtheorem{proposition}[theorem]{Proposition}
\newtheorem{corollary}[theorem]{Corollary}
\newtheorem{definition}[theorem]{Definition}

\newtheorem{problem}[theorem]{Problem}

\newcommand{\wto}{ \ \stackrel{w} {\longrightarrow} \ }

\numberwithin{theorem}{section}
\numberwithin{equation}{section}

\title[A new class evolution multivalued quasi-variational inequalities]{A new class of evolution multivalued quasi-variational inequalities I: existence and nonsmooth optimal control}

\author[S.\,Zeng]{Shengda Zeng}
\address[S.\,Zeng]{Guangxi Colleges and Universities Key Laboratory of Complex System Optimization and Big Data Processing, Yulin Normal University, Yulin 537000, Guangxi, P.R. China \& Jagiellonian University in Krakow,
Faculty of Mathematics and Computer Science, ul. Lojasiewicza 6, 30-348 Krakow, Poland}
\email{zengshengda@163.com}

\author[V.D.\,R\u adulescu]{Vicen\c{t}iu D. R\u{a}dulescu$^{*}$}
\address[V.D.\,R\u adulescu]{Faculty of Applied Mathematics, AGH University of Krak\'ow, 30-059 Krak\'ow, Poland \& Department of Mathematics, University of Craiova, 200585 Craiova, Romania \&	Simion Stoilow Institute of Mathematics of the Romanian Academy, 010702 Bucharest, Romania}
\email{radulescu@inf.ucv.ro}
\thanks{$^{*}$Corresponding author}


\begin{document}

\begin{abstract}
In this paper, we consider a new kind of evolution multivalued quasi-variational inequalities with feedback effect and a nonlinear bifunction which contain several (evolution) quasi-variational/hemivariational inequalities as special cases. The main contribution of this paper is twofold. The first goal is to establish a novel framework for proving the existence of solutions and the compactness of solution set to the evolution multivalued quasi-variational inequalities, under quite mild assumptions. Whereas, the second contribution is to introduce and study a nonlinear and nonsmooth optimal control problem governed by an evolution multivalued quasi-variational inequality, and then to obtain the sufficient conditions for guaranteeing the solvability of the nonlinear and nonsmooth optimal control problem under consideration. Such  nonlinear and nonsmooth optimal control problem could as a useful model to 
 explore the simultaneous distributed-boundary optimal control problems driven by evolution multivalued quasi-variational inequalities, and optimal parameters identification for evolution multivalued quasi-variational inequalities.
\end{abstract}
	
\maketitle

\section{Introduction}

After the pioneering work of Stampacchia~\cite{Stampacchia}, Hartman-Stampacchia~\cite{Hartman-Stampacchia}, Lions-Stampacchia~\cite{Lions-Stampaccia}, and Br\'ezis~\cite{BrezisJMPA} (who firstly used the properties of maximal monotone operators to handle with a linear evolutionary variational inequality), the reserach of variational inequalities has attracted plenty of attention, because variational inequality can be a useful and powerful mathematical models and tools to study various complicated engineering problems and physics processes, and to depict abnormal natural phenomena and mechanism of economic decision.  
Typically speaking, variational inequalities could be classified into two types: 
\begin{itemize}
	\item[$\bullet$] stationary variational inequalities which contain elliptic variational inequalities and time-dependent variational inequalities without time derivative operators;
	\item[$\bullet$] evolution variational inequalities which include parabolic variational inequalities, hyperbolic varriational inequalities and variational inequalities with time fractional-order  derivative operators. 
\end{itemize}
 With the development of theory and applications of variational analysis, nonsmooth analysis and partial differential equations, various variational inequalities and its extends are brought to our attention, see for example,  Hinterm\"uller el al.~\cite{Hintermuller-Rautenberg-SIOPT-2013,Hintermuller2,Hintermuller3} (optimal control and optimal shape design for elliptic variational inequalities), Kikuchi-Oden~\cite{Kikuchi-Oden} (finite element methods for variational inequalities), Han et al.~\cite{Han1,Han2,Han3} (variational inequalities arising in viscoelastic contact  problems), Pang et al.~\cite{Pang1,Pang2,Pang3} (differential variational inequalities, and parameteric variational inequalities), Chen et al.~\cite{Chen1,Chen2,Chen3} (stochastic variational inequalities) and so on.

It is well-known that the constraint sets of variational inequalities or generalized variational inequalites (for instance hemivariational inequalities) are independent of the unknown solution. Whereas, there are many comprehensive and crucial real problems arising in physical and economic models as well as the industrial production (for example, differential Games with incomplete information and shared resource (see Chan-Pang~\cite{Chan-Pang1982} and Han-Pang~\cite{Han-PangMP2010}), $(s,S)$ policy with compound Poisson and diffusion demands (see Bensoussan et al.~\cite{Bensoussan1,Bensoussan2}), and Bean-Kim superconductivity model with temperature and magnetic field  (see Yousept~\cite{YouseptESAIM})), which are modeled eventually for the variational inequalities in which the constraint sets are required to rely explicitly on the unknown solution. This motivates us to introduce the notion of quasi-variational inequalities which can be seen a class of generalized variational inequalities such that the constraints sets depended on the unknown solutions. Because of the dependence of unknown solutions for the constraints, this leads to the difficulty that the classical surjectivity theorems, the method of super-solutions and sub-solutions, and   variational analysis are in-available (or can not be applied directly) for quasi-variational inequalities. Therefore, some novel, effective and useful methods and techniques are being proposed and introduced. We mention some recent representative researches:
Facchinei-Kanzow-Sagratella~\cite{Facchinei-Kanzow-Sagratella} applied a globally convergent algorithm based on a potential reduction technique and  Karush-Kuhn-Tucker conditions to study a quasi-variational inequalities in finite dimensional spaces;  Alphonse-Hinterm\"uller-Rautenberg~\cite{Alphonse-Hintermuller-Rautenberg1,Alphonse-Hintermuller-Rautenberg2,Alphonse-Hintermuller-Rautenberg3} systemically gave the directional differentiability of the solution map of  elliptic and parabolic quasi-variational inequalities of obstacle type type, and proved the stability of solution sets for the  quasi-variational inequalities of obstacle type; in an infinite dimensional Banach space setting,  Kanzow-Steck~\cite{Kanzow1,Kanzow2}  established a theoretical framework for the analysis of existence of solutions to a class of quasi-variational inequalities, 
and then they employed the theoretical results to construct a useful and impressive approximating algorithm to the quasi-variational inequalities in which the main method is based on the augmented Lagrangian and exact penalty approaches as well as the theory of pesuedomonotone operators; Chen-Wang~\cite{Chen-Wang-MP} considered a complicated dynamics system, which is composed of a ordinary differential equation and a time-dependent quasi-variational inequality, arising in the differential Nash equilibrium problems with shared constraints, and proposed a regularized smoothing method to find a solution of such dynamics system; by using the semigroup theory for Maxwell's equations, Yosida regularization,  the subdifferential calculus and semidiscrete Ritz-Galerkin approximation technique, Yousept~\cite{YouseptESAIM,Yousept2,Yousept3} used the idea of quasi-variational inequalities to study various evolution Maxwell equations with Bean's constitutive law between the electric field and the current density. For more details on this direction, the reader is referred to 
Cubiotti-Yao~\cite{Cubiotti-Yao}, Zeng-Khan-Mig\'orski~\cite{Zeng-Khan-MigorskiSCM,Zeng-Migorski-Khan-SICON}, Aussel et al.~\cite{Aussel1,Aussel2}, Adly el al.~\cite{Tahar1,Tahar2}, Kubo-Yamazaki~\cite{Kubo-Yamazaki} and the references
therein.

Although there are a large number of publications concerning the theory, numerical analysis and applications of quasi-variational inequalities. Few have been dedicated to the quasi-variational inequalities with noncovex and nonmonotone framework or evolution quasi-variational inequalities with feedback effect, see~\cite{Barrett-Prigozhin,Cen-Khan-Motreanu-Zeng,Daniele-EJOR,Miranda-Rodrigues-Santos,Nguyen-Qin-SVAA2020,Pang-Fukushima2005,Rodrigues-Santos-AMO}. Recently, Khan-Mig\'orski-Zeng~\cite{Khan-Migorski-Zeng-Optimization2024} introduced a sort of evolutionary quasi-hemivariational inequalities which contain convex and nonconvex potentials, and proved the existence of a solution under strict growth conditions and compact assumptions. This limits the scope of the applications to  evolutionary quasi-variational inequalities. The first novelty of this paper is to remove these flaws and to develop a new and generalized theoretical framework for proving the sovability of a kind of complicated evolution quasi-variational inequalities involving a multivalued operator (which is nonmonotone and can be seen as a multivalued feedback term) and a nonlinear bifunction as follows: 
\begin{problem}\label{problems1}
	Given an element $\mathscr E\in X^*$, find $x\in \mathscr M(x)\cap D(\mathscr L)$ and $\xi\in \mathscr G(\gamma x)$ such that 
	\begin{eqnarray}\label{eqns3.1}
		\langle\mathscr L(x)+\mathscr F(x)-\mathscr E,y-x\rangle_X+\langle\xi,\gamma(y-x)\rangle_Y\ge \Psi(x,x)-\Psi(x,y) 
	\end{eqnarray} 	
	for all $y\in \mathscr M(x)\cap D(\mathscr L)$.
\end{problem}
\noindent Here, $(X,\|\cdot\|_X)$ and $(Y,\|\cdot\|_Y)$ are  reflexive and separable Banach spaces with its dual spaces $(X^*,\|\cdot\|_{X^*})$ and $(Y^*,\|\cdot\|_{Y^*})$, $\mathscr L\colon D(\mathscr L)\subset X\to X^*$ is a linear maximal monotone operator, $\mathscr F\colon X\to X^*$ is a  fully nonlinear operator, $\Psi\colon X\times X\to \mathbb R$ is a nonlinear bifunction, and $\mathscr M\colon X\to 2^X$ and $\mathscr G\colon Y\to 2^{Y^*}$ are two given multivalued mappings.  Whereas, the second goal of this paper is to introduce and study a nonlinear and nonsmooth optimal control problem driven by evolution multivalued quasi-variational inequality, Problem~\ref{problems1}, in which   $\mathscr E\in X^*$ is a control variable, and the fully nonlinear operator $\mathscr F$ and multivalued mapping $\mathscr G$ are described  (or are determinated) by control variables (or parameters) $e\in \Pi$ and $l\in \Theta$, namely, $\mathscr F\colon \Pi\times X\to X^*$ and $\mathscr G\colon \Theta\times Y\to 2^{Y^*}$. Because $(e,l,\mathscr E)\in \Pi\times \Theta\times X^*$ can be seen as control variables or unknown parameters (which could be discontinuous), so, the nonlinear and nonsmooth optimal control problem under consideration can be applied to study simultaneous distributed-boundary optimal control
problems driven by evolution multivalued quasi-variational inequalities, optimal parameters identification for evolution multivalued quasi-variational inequalities, and so forth. This is the second motivation of the present paper. Moreover, it should be mentioned that the theoretical results established in this paper could be applied to explore various parabolic differential inclusions with nonlinear partial differential operators, semipermeability problems with mixed boundary conditions, and non-stationary Non-Newton fluid problems with multivalued and nonmonotone friction law, and so on (more details, one could refer our second paper~\cite{Zeng-Radulescu}).

We end this section by  recalling some necessary preliminary material including convex analysis, variational inequalities and nonsmooth analysis which will be used in next sections for establishing the main results of this paper. 

In the sequel, the symbol ``$\to$" (resp. ``$\wto$") is used to represent the strong (resp. weak) convergence  in various norm spaces. Given a real and reflexive Banach space $(V,\|\cdot\|_V)$, we say that function  $j\colon V\to \mathbb R$ is locally Lipschitz continuous at $x\in V$, if we can find a constant $c_x>0$ and a neighborhood $O(x)$ of $x$ satisfying
\begin{equation*}
	|j(y)-j(z)|\le c_x\|y-z\|_V \ \ \mbox{for all}\ \ y, z\in O(x).
\end{equation*}
\begin{definition}\label{SUB}
	Assume that $j\colon V\to \mathbb R$ is a locally Lipscthiz function on $D\subset V$, 
	the generalized (Clarke) directional derivative of $j$ at the point $x\in V$ in the direction $y\in V$, denoted by $j^0 (x; y)$, is given by 
	\begin{equation*}\label{defcalark}
		j^0(x;y) = \limsup
		\limits_{t\to 0^{+}, \, z\to x} \frac{j(z+t y)-j(z)}{t}.
	\end{equation*}
	The generalized subdifferential operator for $j$ at $x\in V$ is defined  by
	\begin{equation*}
		\partial_C j(x) = \{\, \eta\in V^{*} \,\mid\, j^0 (x; y)\ge
		\langle\eta, y\rangle_{V^*\times V} \ \ \mbox{\rm for all} \ \ y \in V \, \}.
	\end{equation*}
\end{definition}

In the monograph~\cite{Clarke}, we could find a plenty of  impressive and critical properties to the generalized subgradient and generalized directional derivative for  locally Lipschitz functions. Next, we deliver several useful results of generalized subgradient and generalized directional derivative, which will be applied to obtain the main results of this paper.  \begin{proposition}\label{subdiff}
	Given a locally Lipschitz function $j \colon V \to \mathbb R$, we have the following results:	\begin{itemize}
		\item[{\rm(i)}] for every $y \in V$, it is valid	$j^0(x; y) = \max \, \{ \, \langle \eta, y \rangle_{V^*\times V}
		\mid \eta \in \partial_C j(x) \, \}$.
		\item[{\rm(ii)}]  $V \times V \ni (x,y)\mapsto j^0(x;y)
		\in \mathbb R$ is upper semicontinuous.
		\item[{\rm(iii)}] $\partial_Cj\colon V\to 2^{V^*}$ is upper semicontinuous from the strong topology of $V$ to the weak$^*$ topology of $V^*$.  
	\end{itemize}
\end{proposition}

Moreover, we recall the following proposition which gives the sufficient conditions for proving that a variational inequality with a maximal monotone operator and a proper convex functional is solvable. 

\begin{proposition}\label{Proexistence}
	Let $X$ be a real reflexive Banach space with its dual space $X^*$. If the following conditions hold:
	\begin{itemize}
		\item[{\rm(i)}] $A\colon D(A)\to X^*$ is maximal monotone;
		\item[{\rm(ii)}] $\varphi\colon X\to \overline{\mathbb R}:=(-\infty,+\infty]$ is lower semicontinuous and convex with $\varphi\not\equiv+\infty$ (i.e., the effective domain of $\varphi$ is nonempty);
		\item[{\rm(iii)}] one of the condition is true  $\mbox{int}D(A)\cap D(\partial \varphi)$ or $D(A)\cap\mbox{int}D(\partial \varphi)$;
		\item[{\rm(iv)}] $A\colon D(A)\cap D(\varphi)\subset X\to 2^{X^*}$ is coercive in the following sense,  there exist $u_0\in D(A)\cap D(\partial \varphi)$ and $r>0$ such that 
		\begin{eqnarray*}
			\langle u^*,u-u_0\rangle_X>0\mbox{ for all $u^*\in A(u)+\partial \varphi(u)$ with $\|u\|>r$}. 
		\end{eqnarray*}
	\end{itemize}
	Then, the variational inequality has at least one solution 
	\begin{eqnarray*}
		\langle Au,v-u\rangle+\varphi(v)-\varphi(u)\ge 0\mbox{ for all $v\in X$}. 
	\end{eqnarray*}
\end{proposition}

Finally, we recall the well-known Kluge's fixed point theorem for multivalued functions which will be applied to establish the existence of solutions for the evolution multivalued quasi-variational inequality under consideration. 
\begin{theorem}\label{KlugeFPT}
	Suppose $\mathscr D$ be a nonempty, bounded and closed subset of real reflexive Banach space $V$, and $\mathcal Q\colon \mathscr D\to 2^{\mathscr D}$ is a multivalued function that satisfies the following conditions: 
	\begin{itemize}
		\item[{\rm(i)}] for every $u\in \mathscr D$, the set $\mathcal Q(u)$ is nonempty, closed and convex;
		\item[{\rm(ii)}] the graph Gr($\mathcal Q$) of $\mathcal Q$ is weakly closed.
	\end{itemize}
	Then,  the fixed point set of $\mathcal Q$ is nonempty.
\end{theorem}

The remainder of the paper is organized as follows. In   Section \ref{Section3} is concerned with the study of the properties of solution set for the evolution multivalued quasi-variational inequality, Problem~\ref{problems1}, in which we prove the existence and compactness results. As a byproduct, we also give several important existence theorems for some particular cases of Problem~\ref{problems1}.
However, in Section \ref{Section4}, we pay our attention to consider a nonlinear and nonsmooth optimal control problem driven by  evolution multivalued quasi-variational inequality, Problem~\ref{problems1}, and to established the existence theorem for the nonlinear and nonsmooth optimal control problem. It should be mentioning that such nonlinear and nonsmooth optimal control problem could be a model to finding the optimal un-known parameters for evolution multivalued quasi-variational inequalities, when we have measured some data for the solutions of  evolution multivalued quasi-variational inequalities in advance.  Finally, in Section~\ref{Section5}, an inclusion is provided.

\section{Existence and compactness results}\label{Section3}
The section is concerned with the study of the properties of the set of solutions, which contains the nonemptiness and compactness, to the evolution multivalued quasi-variational inequality, Problem~\ref{problems1}. 

To this end, we suppose that the data of Problem~\ref{problems1} fulfill the following conditions: 

\begin{enumerate}
		\item[\textnormal{H($\mathscr K$):}] $\mathscr K$ is a nonempty, closed and convex set of a Banach space $X$.
	\item[\textnormal{H($\mathscr L$):}] $\mathscr L\colon D(\mathscr L)\subset X\to X^* $ is a linear, densely defined and maximal monotone operator. 
	\item[\textnormal{H($\mathscr M$):}] $\mathscr M\colon \mathscr K\to 2^{\mathscr K}$ has nonempty, closed and convex values such that  $0\in \mbox{int}\left(\cap_{w\in \mathscr K}\mathscr M(w)\right)$, and has weakly closed graph with respect to $\mathscr L$, i.e., 
	\begin{enumerate}
		\item[] $x\in \mathscr M(y)\cap D(\mathscr L)$ holds, whenever $\{y_n\},\{x_n\}\subset \mathscr K$ fulfill $x_n\in \mathscr M(y_n)\cap D(\mathscr L)$, $y_n\wto y$   in $\mathcal W$ and $x_n\wto x$ in $\mathcal W$.
	\end{enumerate}
\item[\textnormal{H($\mathscr F$):}] $\mathscr F\colon X\to X^*$ is a bounded, monotone and hemicontinuous function such that there exist constants $c_\mathscr F>0$ and $d_\mathscr F\ge 0$ satisfying 
\begin{eqnarray}\label{ineqF}
	\langle\mathscr F(x),x\rangle_X\ge c_\mathscr F\|x\|_X^p-d_\mathscr F\mbox{ for all $x\in X$},
\end{eqnarray}
with some $1<p<+\infty$. 
\item[\textnormal{H($\mathscr G$):}]  $\mathscr G\colon Y\to 2^{Y^*}$ is a multivalued mapping defined in a reflexive Banach space $Y$ with nonempty, closed and convex values, and  has a sequentially strongly-weakly closed  graph    such that there are constants $c_\mathscr G, d_\mathscr G\ge 0$ fulfilling 
\begin{eqnarray}\label{Ggrowthcondition} 
\|\xi\|_{Y^*}\le c_\mathscr G\|z\|_Y^{p-1}+d_\mathscr G
\end{eqnarray}
for all $\xi\in \mathscr G(z)$ and all $z\in Y$. 
\item[\textnormal{H($\mathscr E$):}] $\mathscr E\in X^*$.
\item[\textnormal{H($\gamma$):}]  $\gamma \colon X\to Y$ is linear and continuous such that $\gamma|_{\mathcal W}$ is compact (i.e., $\gamma$ is compact on $\mathcal W$), where $\mathcal W:=\{x\in X\cap D(\mathscr L)\,\mid\,\mathscr Lx\in X^*\}$.
\item[\textnormal{H($\Psi$):}] $\Psi\colon X\times X\to \mathbb R$ satisfies the following properties: 
\begin{enumerate}
	\item[(i)] for each $x\in X$, the function $\Psi(x,\cdot)\colon X\to \mathbb R$ is convex and lower semicontinuous;
	\item[(ii)] there is a bounded function $b_\Psi\colon  X^3\to \mathbb R_+$ such that  for each bounded set $B\subset X$ it holds
		\begin{eqnarray*}
		\lim_{\|x\|_X\to +\infty}\frac{b_\Psi(z,0,x)}{\|x\|_X^{p-\eta}}=0 \mbox{ uniformly in $z\in B$}
	\end{eqnarray*}
	with $0<\eta<p$ and  
	\begin{eqnarray}\label{inePsi}
		\Psi(x,y_1)-\Psi(x,y_2)\le b_\Psi(x,y_1,y_2)\|y_1-y_2\|_X^\eta
	\end{eqnarray}
	for all $x,y_1,y_2\in X$, where   $\mathbb R_+:=[0,+\infty)$; 
	\item[(iii)] the inequality is valid 
	\begin{eqnarray*}
		\limsup_{n\to\infty}\left(\Psi(w_n,y_n)-\Psi(w_n,x_n)\right)\le \Psi(w,y)-\Psi(w,x),
		\end{eqnarray*} 
		whenever $\{w_n\},\{y_n\},\{x_n\}\subset \mathscr K\cap D(\mathscr L)$ are such that  $w_n \wto w$ and $x_n\wto x$ in $\mathcal W$ (i.e., $x_n\wto x$ and $w_n\wto w$ in $X$, and $\mathscr Lx_n\wto \mathscr Lx$ and $\mathscr Lw_n\wto \mathscr Lw$ in $X^*$) and $y_n\to y$  in $X$;
		\item[(iv)] $0\in D(\partial \Psi(v,\cdot))$ for all $v\in X$ (where $D( \partial \Psi(v,\cdot))$ is the effective domain of function   $X\ni x \mapsto  \partial\Psi(v,x)\subset X^*$) with $|\Psi(v,0)|\le e_\Psi$ for all $v\in X$, and there exist  constants $c_\Psi,d_\Psi\ge 0$ such that  
		\begin{eqnarray*}
			\Psi(v,y)\ge -c_\Psi\|y\|_X^\beta-d_\Psi
			\end{eqnarray*}
			with some $1\le\beta<  p$, for all $v,y\in X$. 
\end{enumerate}
\end{enumerate}

\begin{remark}\label{remarks3.1}
 Because of the reflexivity of $X$, so,  $\mathcal W$ is endowed the graph norm
$$
\|x\|_{\mathcal W}:=\|x\|_X+\|\mathscr Lx\|_{X^*}\mbox{ for all $x\in \mathcal W$}
$$
to be a reflexive Banach space.
	
Indeed, the assumptions imposed above are mild. There are a plenty of functions satisfy these assumptions. Here, we provide several particular cases: 
\begin{itemize}
\item[{\rm (i)}] A classical example for maximal operator $\mathscr L$ is that $\mathscr L=\frac{d }{dt}$, namely, $\mathscr L$ is the time derivative operator in the distribution sense. 
\item[{\rm (ii)}] Let  $1<p<+\infty$ and $T>0$ be fixed, and $\Omega \subset \mathbb R^N$ be a  bounded domain with $N\ge 2$.   We set $ X= L^p(0,T; W^{1,p}_0(\Omega))$. Then, the function $\mathscr F\colon X\to X^*$ defined by
\begin{eqnarray*}
\langle\mathscr F(u),v\rangle_X:=\int_0^T\int_\Omega|\nabla u(x,t)|^{p-2}(\nabla u(x,t),\nabla v(x,t))_{\mathbb R^N}\,dx\,dt\mbox{ for all $u,v\in X$}
\end{eqnarray*} 
enjoys hypothesis H($\mathscr F$). In fact, we have 
\begin{itemize}
\item[$\bullet$] For every $u\in X$, it holds 
\begin{eqnarray*}
	\langle\mathscr F(u),u\rangle_X:=\int_0^T\int_\Omega|\nabla u(x,t)|^{p}\,dx\,dt=\|u\|_X^p\mbox{ for all $u\in X$},
\end{eqnarray*}
thus, inequality (\ref{ineqF}) is true with $c_\mathscr F=1$ and $d_\mathscr F=0$.
\item[$\bullet$] The estimates indicate that $\mathscr F$ is a bounded mapping
\begin{align*}
\|\mathscr F u\|_{X^*}=&\sup_{v\in X,\|v\|_X=1}\langle\mathscr Fu,v\rangle_X \\
=&\sup_{v\in X,\|v\|_X=1}\int_0^T\int_\Omega|\nabla u(x,t)|^{p-2}(\nabla u(x,t),\nabla v(x,t))_{\mathbb R^N}\,dx\,dt\\
\le &\sup_{v\in X,\|v\|_X=1}\left(\int_0^T\int_\Omega|\nabla u(x,t)|^{p}\,dx\,dt\right)^{\frac{p-1}{p}}\left( \int_0^T\int_\Omega|\nabla v(x,t)|^p\,dx\,dt\right)^\frac{1}{p}\\
=&\|u\|_X^{p-1}\mbox{ for all $u\in X$},
\end{align*} 
where the last inequality is obtained by utilizing H\"oder inequality. 
\item[$\bullet$] The monotonicity of $\mathbb R^N\ni\xi\mapsto |\xi|^{p-2}\xi\in \mathbb R^N$ guarantees that $\mathscr F$ is monotone as well. 
\item[$\bullet$] Let $\{u_n\}\subset X$ be such that $u_n\to u$ in $X$. Notice that 
\begin{eqnarray*}
\|\mathscr F u-\mathscr Fu_n\|_{X^*}\le \left(\int_0^T\int_\Omega|(|\nabla u(x,t)|^{p-2}\nabla u-|\nabla u_n|^{p-2}\nabla u_n)|^p\,dx\,dt\right)^{\frac{p-1}{p}},
\end{eqnarray*}
so, it could apply Lebesgue dominated convergence theorem to show that $\mathscr F$ is continuous. 
\end{itemize}
\item[{\rm(iii)}] There are a plenty of multivalued functions which satisfy the hypotheses $H(\mathscr G)$. Here, we point out that when $\mathscr G$ is formulated by the Clarke subdifferential operator of a locally Lipschitz function $J\colon Y\to \mathbb R$, i.e., $\mathscr G=\partial_C J$, in which $\partial_C J$ (the Clarke subdifferential operator of $J$) fulfills the  growth condition
\begin{eqnarray*}
	\|\xi\|_{Y^*}\le c_J\|w\|_Y^{p-1}+d_J
\end{eqnarray*}
for all $\xi\in \partial_C J(w)$ and $w\in Y$. 
Then, the multivalued function $\mathscr G=\partial_C J$ reads all conditions of $H(\mathscr G)$ (see Theorem~\ref{theorems2} below).  
\item[{\rm(iv)}] Let $\mathcal W=\{u\in L^p(0,T;W^{1,p}(\Omega))\,\mid\,u'\in L^q(0,T;W^{-1,p}(\Omega)):=L^p(0,T;W^{1,p}(\Omega))^*\}$ with $\frac{1}{p}+\frac{1}{q}=1$. The bifunction $\Psi\colon X\times X\to  \mathbb R$ defined by
\begin{eqnarray*}
	\Psi(v,u):=\int_{D\times [0,T] }\theta(v(x,t))|u(x,t)|^\beta\,dx\,dt
\end{eqnarray*}
enjoys all properties of H($\Psi$), where 
$X=L^p(0,T;W^{1,p}(\Omega))$ and $D\subset \overline \Omega$ is a nonempty subset with positive measure such that $W^{1,p}(\Omega)$ embeds compactly to $L^p(D)$, and $\theta\colon \mathbb R\to (0,+\infty)$ is Lipschitz continuous with constant $L_\theta>0$ and there are two constants $0<c_\theta\le d_\theta$ such that 
\begin{eqnarray*}
	c_\theta\le \theta(s)\le d_\theta\mbox{ for all $s\in \mathbb R$}. 
\end{eqnarray*}
Indeed, we have 
\begin{itemize}
	\item[$\bullet$] for any $v\in X$ fixed, by the definition of $\Psi$, we can see that $X\ni u\mapsto \Psi(v,u)\in[0,+\infty)$ is convex and continuous.
	\item[$\bullet$] for any $w,y_1,y_2\in X$, if $\beta=1$, we have
		\begin{align*}
		\Psi(w,y_1)-\Psi(w,y_2)=&\int_{D\times [0,T]}\theta(w(x,t))\left(|y_1(x,t)|-|y_2(x,t)|\right)\,dx\,dt\\
		\le &d_\theta\int_{D\times [0,T] } |y_1(x,t)-y_2(x,t)|\,dx\,dt\\
		\le &d_\theta (T|D|)^\frac{1}{p'}C_X(p)\|y_1-y_2\|_X,
		\end{align*}
		namely, inequality (\ref{inePsi}) holds with $b_\Psi\equiv d_\theta (T|D|)^\frac{1}{p'}$  and $\eta=1$, where $C_X(p)>0$ is the smallest constant  such that 
		$$
		\|z_1-z_2\|_{L^p(D\times [0,T])}\le C_X(p) \|z_1-z_2\|_X\mbox{ for all $z_1,z_2\in X$}
		$$
		(this inequality holds by using Aubin's lemma, and Sobolev embedding theorem or trace theorem).
	However, when $1<\beta<p$, one has 
	\begin{align*}
		\Psi(w,y_1)-\Psi(w,y_2)=&\int_{D\times [0,T] }\theta(w(x,t))\left(|y_1(x,t)|^\beta-|y_2(x,t)|^\beta\right)\,dx\,dt\\
		=&\int_{D\times [0,T] } \theta(w(x,t))\beta\zeta(y_1(x,t),y_2(x,t))^{\beta-1}(|y_1(x,t)|-|y_2(x,t)|)\,dx\,dt\\
		\le &\int_{D\times [0,T] } \theta(w(x,t))\beta\zeta(y_1(x,t),y_2(x,t))^{\beta-1}|y_1(x,t)-y_2(x,t)|\,dx\,dt\\
		\le &d_\theta\beta\int_{D\times [0,T] } \zeta(y_1(x,t),y_2(x,t))^{\beta-1}|y_1(x,t)-y_2(x,t)|)\,dx\,dt\\
		\le&d_\theta\beta C_X(\beta)\left(\int_{D\times [0,T] } \zeta(y_1(x,t),y_2(x,t))^\beta\,dx\,dt\right)^{\beta'}\|y_1-y_2\|_X^\beta,
	\end{align*}
	where $\zeta(y_1(x,t),y_2(x,t))\in \left[\min\{|y_1(x,t)|,|y_2(x,t)|\},\max\{|y_1(x,t)|,|y_2(x,t)|\}\right]$. This means that (\ref{inePsi}) is valid with $$b_\Psi(w,y_1,y_2)=d_\theta\beta C_X(\beta)\left(\int_{D\times [0,T] } \max\{|y_1(x,t)|,|y_2(x,t)|\}^\beta\,dx\,dt\right)^{\beta'}$$
	and $\eta=\beta$. 
	\item[$\bullet$] let sequences $\{w_n\},\{y_n\},\{z_n\}\subset X\cap D(\mathscr L)$ be such that $w_n\wto w$ and $z_n\wto z$ in $\mathcal W$, and $y_n\to y$  in $X$. Then, by  Aubin's lemma, the embedding from $\mathcal W$ to $L^p(D\times(0,T))$ is compact.  So, we may assume that $w_n(x,t)\to w(x,t)$, $y_n(x,t)\to y(x,t)$ and $z_n(x,t)\to z(x,t)$ for a.e. $(x,t)\in D\times [0,T]$. Therefore, from the continuity of $\theta$ and Lebesgue Dominated Convergence theorem, it finds 
	\begin{align*}
	&\lim_{n\to\infty}\left(	\Psi(w_n,y_n)-\Psi(w_n,x_n)\right)\\
	=&\lim_{n\to\infty}\left(\int_0^T\int_D\theta(w_n(x,t))\left(|y_n(x,t)|^\beta-|x_n(x,t)|^\beta\right)\,dx\,dt\right)\\
	=&\int_0^T\int_D\lim_{n\to\infty}\theta(w_n(x,t))\left(|y_n(x,t)|^\beta-|x_n(x,t)|^\beta\right)\,dx\,dt\\
	=&\int_0^T\int_D\theta(w(x,t))\left(|y(x,t)|^\beta-|x(x,t)|^\beta\right)\,dx\,dt\\
	=&\Psi(w,y)-\Psi(w,z).
	\end{align*}
	So, hypothesis H($\Psi$)(iii) is available. 
	\item[$\bullet$] by the definition of $\Psi$, we can see that $\Psi(v,u)\ge 0$ for all $v,u\in X$, namely, it could take $c_\Psi=d_\Psi=0$.
\end{itemize}
\item[{\rm(v)}] Let $Y=L^p(\partial \Omega\times (0,T))$, $X=L^p(0,T;W^{1,p}(\Omega))$ and $\mathcal W:=\{u\in X\,\mid\,u'\in X^*\}$. Then, from trace theorem, we can see that $\gamma \colon \mathcal W\to Y$ is compact. 
\end{itemize}
\end{remark}
	
The main result of this section is the following theorem which reveals that the solution set of Problem~\ref{problems1} is nonempty and weakly compact in $\mathcal W$.

\begin{theorem}\label{theorems1}
Under the assumptions of  H($\mathscr K$), H($\mathscr L$), H($\mathscr F$), H($\mathscr G$), H($\gamma$), H($\mathscr E$), H($\Psi$) and H($\mathscr M$), if the inequality holds  
	\begin{eqnarray}\label{smallnesscondition}
	c_\mathscr F>c_\mathscr G\|\gamma\|^p,
\end{eqnarray}
then the solution set of Problem~\ref{problems1} is nonempty and weakly compact in $\mathcal W$. Moreover, if $\mathscr F$ satisfies $(S_+)$-property with respect to $\mathscr L$, namely, 
\begin{itemize}
	\item[] if $u_n\wto u$ in $\mathcal W$ and $\limsup_{n\to\infty}\langle\mathscr F(x_n),x_n-x\rangle_X\le 0$, then we have $x_n\to x$ in $X$, 
\end{itemize}
then the solution set of Problem~\ref{problems1} is compact in $X$ as well.
	\end{theorem} 

To establish the existence of solutions for Problem~\ref{problems1}, let us  consider the following parameter evolution variational inequality: 
\begin{problem}\label{problems2}
	Given elements $(z,\xi,\mathscr E)\in \mathscr K\times Y^*\times X^*$, find $x\in \mathscr M(z)\cap D(\mathscr L)$ such that 
	\begin{eqnarray}\label{eqns3.4}
		\langle\mathscr L(x)+\mathscr F(x)-\mathscr E,y-x\rangle_X+\langle\xi,\gamma(y-x)\rangle_Y\ge \Psi(z,x)-\Psi(z,y) 
	\end{eqnarray} 	
	for all $y\in \mathscr M(z)\cap D(\mathscr L)$.
	\end{problem}

It could apply Proposition~\ref{Proexistence}  to conclude that Problem~\ref{problems2} is solvable.
	
\begin{lemma}\label{lemmas1}
	Under the assumptions of Theorem~\ref{theorems1}, for each $(z,\xi)\in \mathscr K\times Y^*$ the solution set of Problem~\ref{problems2}, denoted by $\mathscr S(z,\xi)$, is nonempty in $\mathcal W$. Moreover, when $\mathscr F$ is strictly monotone, $\mathscr S(z,\xi)$ is singleton.  
\end{lemma}

\begin{proof}
For any $(z,\xi)\in \mathscr K\times Y^*$  fixed, let $\varphi\colon X\to \overline{\mathbb R}$ be defined by 
\begin{eqnarray*}
	\varphi(x):=\Psi(z,x)+I_{\mathscr M(z)}(x)\mbox{ for all $x\in X$}, 
\end{eqnarray*}
where $I_{\mathscr M(z)}$ is the indicator function of set $\mathscr{M}(z)$. It is obvious that $\varphi$ is proper (indeed, $D(\varphi)=\mathscr M(z)$),  convex and lower semicontinuous. Using this function, we can take a standard procedure to obtain that $x\in \mathscr K$ is a solution of Problem~\ref{problems2}, if and only if, it solves the following inclusion with the sum of maximal monotone operators 
\begin{eqnarray}\label{eqns3.5}
	\mathscr Lx+\mathscr F(x)+\partial \varphi(x)\ni \mathscr E-\gamma^*\xi.
	\end{eqnarray}
Next, we are going to employ Proposition~\ref{Proexistence} to show that $\mathscr S(z,\xi)\neq\emptyset$.  
So, we shall verify that the operator $A\colon X\to X^*$ defined by 
\begin{eqnarray*}
	Ax:=\mathscr Lx+\mathscr Fx\mbox{ for all $x\in X$}
	\end{eqnarray*}
and the function $\varphi\colon X\to \overline{\mathbb R}$  
verify all conditions of Proposition~\ref{Proexistence}.
\begin{itemize}
	\item[$\bullet$] {\bf $A$ is maximal monotone.} Recall that $\mathscr F\colon X\to X^*$ is hemicontinuous, monotone and bounded, and $\mathscr L$ is maximal monotone, so, by virtue of the use of Corollary 1.1 of Barbu~\cite[p. 44]{Barbu1985}, we can see that $A$ is maximal monotone too.
	 	\item[$\bullet$] {\bf $\mbox{int}D(A)\cap D(\partial \varphi)\neq\emptyset$.} From hypothesis H($\mathscr L$) and the conditions $0\in \mbox{int}\left(\cap_{w\in \mathscr K}\mathscr M(w)\right)$ as well as $0\in D(\partial \Psi(z,\cdot))$ for all $z\in X$, we can see that $0\in \mbox{int}D(A)\cap D(\partial \varphi)$, namely,  $\mbox{int}D(A)\cap D(\partial \varphi)\neq\emptyset$.
	 \item[$\bullet$] {\bf $A+\partial \varphi$ is coercive.} Recall that $0\in D(A)\cap D(\partial \varphi)$, so, for any $(x,\zeta)\in Gr(A+\partial \varphi)\subset X\times X^*$, it can calculate by using monotonicity of $\mathscr L$ and inequality (\ref{ineqF}) that 
	 \begin{align*}
	 	\langle Ax+\zeta,x\rangle_X=&\langle \mathscr Lx+\mathscr F(x)+\zeta,x\rangle_X\\
	 	\ge&c_\mathscr F\|x\|_X^p-d_\mathscr F+ \varphi(x)-\varphi(0)\\
	 	\ge &c_\mathscr F\|x\|_X^p-d_\mathscr F-c_\varphi\|x\|_X-d_\varphi-\varphi(0)
	 \end{align*}
	 for some constants $c_\varphi,d_\varphi\ge 0$, 
	 where the last inequality is obtained by applying the well-known result that there exists an affine function which is less than function $\varphi$ (see for example, Proposition 1.10 of~\cite{Brezis-2011}). Then, it yields
	 \begin{eqnarray*}
	 	\lim_{\|x\|_X\to +\infty}	\frac{\langle Ax+\zeta,x\rangle_X}{\|x\|_X}=+\infty.
	 \end{eqnarray*}
	 This reveals that there is a constant $r_0>0$ such that for all $\|x\|_X>r_0$ with $x\in D(A)\cap D(\partial \varphi)$ the inequality is available
	 \begin{eqnarray*}
	 	\langle x^*,x\rangle_X>0\mbox{ for all $x^*\in A(x)+\partial \varphi(x)$},
	 \end{eqnarray*}
	 namely, $A+\partial \varphi$ is coercive.
\end{itemize} 
Now, it can invoke Proposition~\ref{Proexistence}  to admit that inclusion (\ref{eqns3.5}) is solvable, thus, $\mathscr S(z,\xi)$ is nonempty in $\mathcal W$. 

Moreover, suppose that $\mathscr F$ is strictly monotone and $x_1,x_2 \in \mathscr S(z,\xi)$. Let us put $y=x_2$ into (\ref{eqns3.4}) with $x=x_1$, and take $y=x_1$ in (\ref{eqns3.4}) with $x=x_2$. We sum up the resulting inequalities to get 
\begin{eqnarray*}
	\langle \mathscr L(x_1)-\mathscr L(x_2)+\mathscr F(x_1)-\mathscr F(x_2),x_1-x_2\rangle_X\le 0.  
\end{eqnarray*}
We infer that $x_1=x_2$. 
	\end{proof}
	
	Using this proposition, we could observe that $\mathscr S\colon \mathscr K\times Y^*\to  \mathcal W\cap\mathscr K$ is well-defined, when $\mathscr F$ is strictly monotone. In fact, $\mathscr S$ is usually called as parameter variational selection of Problem~\ref{problems2}.  Moreover,  we shall prove several critical properties of  $\mathscr S\colon \mathscr K\times Y^*\to \mathcal W\cap\mathscr K$. 
	
\begin{lemma}\label{lemmas2}
Assume that all conditions of Lemma~\ref{lemmas1} including the strict monotonicity of $\mathscr F$ are fulfilled. Then, $\mathscr S\colon \mathscr K\times Y^*\to \mathcal W\cap \mathscr K$ is a bounded function.
	\end{lemma}
	\begin{proof}
		We first prove that $\mathscr S$ maps the bounded sets of $\mathscr K\times Y^*$ to bounded sets of $X$.  If the assertion is not true, then there exists a bounded set $\mathcal D\subset \mathscr K\times Y^*$ (i.e., there is a constant $c_\mathcal D>0$ such that $\|(z,\xi)\|_{X\times Y^*}:=\|z\|_X+\|\xi\|_{Y^*}\le c_\mathcal D$ for all $(z,\xi)\in \mathcal D$) such that we could take a sequence $\{(z_n,\xi_n)\}\subset \mathcal D$ satisfying $\|\mathscr S(z_n,\xi_n)\|_X\to+\infty $ when $n\to\infty$. For every $n\in\mathbb N$,  let $x_n=\mathscr S(z_n,\xi_n)$. Owing to $0\in D(\mathscr L)\cap \mathscr M(x)$ for each $x\in\mathscr K$, it takes $y=0$ as test function in  (\ref{eqns3.4}) with $x=x_n$ and $\xi=\xi_n$ to find
	\begin{align*}
			\langle\mathscr L(x_n)+\mathscr F(x_n)-\mathscr E,x_n\rangle_X+\langle\xi_n,\gamma x_n\rangle_Y\le \Psi(z_n,0)-\Psi(z_n,x_n).
	\end{align*}
By virtue of the use of hypotheses H($\mathscr L$), H($\Psi$)(ii) and H($\mathscr F$), we could find a bounded function $b_\Psi\colon X\times X\times X\to \mathbb R_+$ and $0<\eta<p$ such that  
\begin{eqnarray*}
		\Psi(z_n,0)-\Psi(z_n,x_n)\le b_\Psi(z_n,0,x_n)\|x_n\|_X^\eta,
	\end{eqnarray*}
	and 
	\begin{eqnarray*}
	\langle\mathscr L(x_n)+\mathscr F(x_n),x_n\rangle_X\ge c_\mathscr F\|x_n\|_X^p-d_\mathscr F.
	\end{eqnarray*}
	Hence, 
	\begin{align*}
		&	b_\Psi(z_n,0,x_n)\|x_n\|_X^\eta\\
		\ge &\Psi(z_n,0)-\Psi(z_n,x_n)\\
		\ge &	\langle\mathscr L(x_n)+\mathscr F(x_n)-\mathscr E,x_n\rangle_X+\langle\xi_n,\gamma x_n\rangle_Y\\
		\ge&c_\mathscr F\|x_n\|_X^p-d_\mathscr F-\left(\|\mathscr E\|_{X^*}+\|\gamma^*\xi_n\|_{X^*}\right)\|x_n\|_X.
	\end{align*}
Recall that $\{z_n\}$ is bounded in $X$ and $\|x_n\|_X\to \infty$, the identity holds (see hypothesis H($\Psi$)(ii)) 
 \begin{eqnarray*}
	\lim_{n\to +\infty}\frac{b_\Psi(z_n,0,x_n)}{\|x_n\|_X^{p-\eta}}=0.
\end{eqnarray*}
Therefore, we get 
\begin{align*}
	0=&\lim_{n\to\infty}	\frac{b_\Psi(z_n,0,x_n)}{\|x_n\|_X^{p-\eta}}\\
	\ge &c_\mathscr F-\lim_{n\to\infty}\frac{d_\mathscr F}{\|x_n\|_X^{p}}-\lim_{n\to\infty}\frac{\left(\|\mathscr E\|_{X^*}+\|\gamma^*\xi_n\|_{X^*}\right)}{\|x_n\|_X^{p-1}}\\
	=&c_\mathscr F>0.
	\end{align*}
	This leads to a contradiction. So, $\mathscr S$ is bounded from $\mathscr K \times Y^*$ to $\mathscr K$. 
	
We shall show that $\mathscr S$ maps bounded sets of $\mathscr K  \times Y^*$ to bounded sets of $\mathcal W$. By the definition of norm of $\mathcal W$, it is sufficient to verify that $\mathscr L(\mathscr S (\mathcal D))$ is bounded in $X^*$ for each bounded set $\mathcal D$ in $\mathscr K\times Y^*$. From the above analysis, it could find a constant $d_\mathcal D>0$ such that 
$$\|\mathscr L(\mathscr S(\mathcal D))\|_X:=\sup_{x\in \mathscr L(\mathscr S(\mathcal D))}\|x\|_X\le d_\mathcal D.$$  
Remembering  $0\in \mbox{int}\left(\cap_{w\in \mathscr K}\mathscr M(w)\right)$, so, we can pick up an open ball $O_X(0,d_0)$ with $d_0>0$ such that $O_X(0,d_0)\subset \mathscr M(z)$ for all $z\in \mathscr K$. 
Then, (\ref{eqns3.4}) implies 
	\begin{eqnarray*} 
	\langle\mathscr L(x),x-y\rangle\le\langle \mathscr F(x)-\mathscr E,y-x\rangle_X+\langle\xi,\gamma(y-x)\rangle_Y- \Psi(z,x)+\Psi(z,y) 
\end{eqnarray*} 	
for all $y\in O_X(0,d_0)$. Using the monotonicity of $\mathscr L$ and hypothesis H($\Psi$)(ii), one has 
	\begin{align*} 
	\langle\mathscr L(x),-y\rangle&\le\left[\|\mathscr F(x)\|_{X^*}+\|\mathscr E\|_{X^*}+\|\gamma^*\xi\|_{X^*}\right]\|y-x\|_X+b_\Psi(z,x,y)\|x-y\|_X \\
	\le& \left[\|\mathscr F(x)\|_{X^*}+\|\mathscr E\|_{X^*}+\|\gamma^*\xi\|_{X^*}\right](d_\mathcal D+d_0) +b_\Psi(z,x,y)(d_\mathcal D+d_0) \\
	\le& \left[d_\mathscr F+\|\mathscr E\|_{X^*}+\|\gamma^*\xi\|_{X^*}\right](d_\mathcal D+d_0) +l_\Psi(d_\mathcal D+d_0),
\end{align*}
where $d_\mathscr F:=\sup_{x\in O_X(0,d_\mathcal D)}\|\mathscr F(x)\|_{X^*}$ and $l_\Psi:=\sup_{z\in P_\mathscr K(\mathcal D),x\in O_X(0,d_\mathcal D),y\in O_X(0,d_0)}b_\Psi(z,x,y)$ and $P_\mathscr K$ is defined by $P_\mathscr K((x,\xi))=x$ for all $(x,\xi)\in \mathcal D$.
Therefore, we have 
\begin{align*}
 \|\mathscr Lx\|_{X^*}&=\sup_{y\in O_X(0,d_0)}\frac{1}{d_0}\langle\mathscr L(x),-y\rangle\\
 &\le\frac{1}{d_0}\left[d_\mathscr F+\|\mathscr E\|_{X^*}+\|\gamma^*\xi\|_{X^*}+d_\Psi\right](d_\mathcal D+d_0).
\end{align*}
This shows that $\mathscr S$ maps bounded sets of $\mathscr K \times Y^*$ to bounded sets of $\mathcal W$. 
	\end{proof}

Moreover, given a constant $r>0$, let us introduce the 
truncation mapping $\mathcal T_r\colon X\to X$ defined by 
\begin{eqnarray*}
	\mathcal T_r(x):=\left\{\begin{array}{ll}
x&\mbox{ if $\|x\|_X\le r$},\\
\frac{x}{\|x\|_X}&\mbox{ otherwise}.
		\end{array}\right.
\end{eqnarray*}

\begin{lemma}\label{lemmas3}
	Suppose all assumptions of Lemma~\ref{lemmas2} hold.
	Then there exists a constant $c_0>0$ such that 
	\begin{eqnarray*}
		\|x\|_X\le c_0
	\end{eqnarray*}
	for all $x\in \mathscr S (B_X(0,c_0),\mathscr G(\gamma \mathcal T_{c_1}(B_X(0,c_0))))$ for some $c_1\ge c_0$, where $B_X(0,c_0):=\{x\in X\,\mid\,\|x\|_X\le c_0\}$. 
\end{lemma}

\begin{proof}
Let $c_0>0$ be fixed which will be determinated latter. For every $x,z\in B_X(0,c_0)$ and $\xi\in \mathscr G(\gamma T_{c_1}(z))$, we use hypothesis H($\mathcal G$) to attain 
\begin{eqnarray}\label{eqns3.6}
		\|\xi\|_{Y^*}\le c_\mathscr G\|\gamma\mathcal T_{c_1}(z)\|_Y^{p-1}+d_\mathscr G\le  c_\mathscr G\|\gamma\|^{p-1} \|\mathcal T_{c_1}(z)\|_X^{p-1}+d_\mathscr G.
\end{eqnarray}
From the above inequalities, if we suppose that $\|\xi\|_{Y^*}\le c_2$ for all $\xi \in \mathscr G(\gamma T_{c_1}(z))$ and all $z\in B_X(c_0,0)$ and set $x=\mathscr S(z,\xi)$, then it could take $y=0$ in (\ref{eqns3.4}) and use hypothesis H($\Psi$)(iv) to having 
	\begin{align*}
& \left(\|\mathscr E\|_{X^*}+\|\gamma\|c_2\right)\|x\|_X\\
\ge 	& \left(\|\mathscr E\|_{X^*}+\|\gamma^*\xi\|_{X^*}\right)\|x\|_X\\
	\ge &-\langle\xi,\gamma x\rangle_Y-\langle\mathscr E,x\rangle_X\\
	\ge &	\langle\mathscr L(x)+\mathscr F(x) ,x\rangle_X-\Psi(z,0)+\Psi(z,x)\\
	\ge&c_\mathscr F\|x\|_X^p-d_\mathscr F-\Psi(z,0)+\Psi(z,x)\\
	\ge &c_\mathscr F\|x\|_X^p-d_\mathscr F-c_\Psi\|x\|_X^\beta-d_\Psi-e_\Psi,
\end{align*}
namely,
 	\begin{align}\label{est1}
 	& \left(\|\mathscr E\|_{X^*}+\|\gamma\|c_2\right)\|x\|_X +c_\Psi\|x\|_X^{\beta}+d_\mathscr F+d_\Psi+e_\Psi\ge c_\mathscr F\|x\|_X^{p}.
 	\end{align}
 	
 	Let us consider two cases that 
 	\begin{itemize}
 		\item[$\bullet$] If it is true that $\|\mathscr S(z,\xi)\|_X\le 1$ for all $(z,\xi)\in X\times Y^*$, then we can take directly $c_0=c_1=1$. Also, we set $c_2= c_\mathscr G\|\gamma\|^{p-1} +d_\mathscr G.$
 		
 	\item[$\bullet$] When $(z,\xi)\in \mathcal D$ is such that $\|\mathscr S(z,\xi)\|_X>1$. 
 	Keeping in mind that $c_\mathscr F>c_\mathscr G\|\gamma\|^p$ (see (\ref{smallnesscondition})) and (\ref{est1}), we deduce 
 	\begin{align*} 
 		c_\mathscr F\|x\|_X^{p}\le& \left(\|\mathscr E\|_{X^*}+\|\gamma\|c_2\right)\|x\|_X +c_\Psi\|x\|_X^{\beta}+d_\mathscr F+d_\Psi+e_\Psi \\
 		\le&\left(\|\mathscr E\|_{X^*}+\|\gamma\|c_2+d_\mathscr F+d_\Psi+e_\Psi \right)\|x\|_X +c_\Psi\|x\|_X^{\beta}.
 	\end{align*}
 	It could employ Young inequality to find a constant $c_3>0$ such that 
 	\begin{eqnarray}\label{eqns3.8}
 		\frac{ \left(\|\mathscr E\|_{X^*}+\|\gamma\|c_2+d_\mathscr F+d_\Psi+e_\Psi\right) +c_3}{m_0}\ge \|x\|_X^{p-1} 
 		\end{eqnarray}
 		with $m_0:=\frac{c_\mathscr F+c_\mathscr G\|\gamma\|^p}{2}$, where $c_3$ only relies  on  $m_0$. 
 		
 		Next, we are going to verify that the constant $c_0$ could be taken by 
 		\begin{eqnarray}\label{eqns3.9}
 				c_0:=	\left(\frac{ \left(\|\mathscr E\|_{X^*}+\|\gamma\|c_2+d_\mathscr F+d_\Psi+e_\Psi\right) +c_3}{m_0}\right)^\frac{1}{p-1}.
 		\end{eqnarray}
 	So, we have to calculate $c_2$ in advance. We take $c_1=c_0$. Indeed, from (\ref{eqns3.8}) and (\ref{eqns3.6}), we have 
 	\begin{align*}
 		 c_\mathscr G\|\gamma\|^{p-1} \|\mathcal T_{c_1}(z)\|_Y^{p-1}+d_\mathscr G&\le c_\mathscr G\|\gamma\|^{p-1} c_0^{p-1}+d_\mathscr G\\
 		 &\le	\frac{  c_\mathscr G\|\gamma\|^{p-1} \left(\|\mathscr E\|_{X^*}+\|\gamma\|c_2+d_\mathscr F+d_\Psi+e_\Psi\right) +c_3}{m_0}+d_\mathscr G.
 		\end{align*}
It is not hard to prove that for each 
$c_2\ge c_4$ with 
$$
c_4:=\frac{ c_\mathscr G\|\gamma\|^{p-1}(\|\mathscr E\|_{X^*}+d_\mathscr F+d_\Psi+e_\Psi)+c_3+m_0d_\mathscr G}{m_1}
$$
with 
$m_1:=\frac{c_\mathscr F-c_\mathscr G}{2}$, the inequality holds
$$
\frac{  c_\mathscr G\|\gamma\|^{p-1} \left(\|\mathscr E\|_{X^*}+\|\gamma\|c_2+d_\mathscr F+d_\Psi+e_\Psi\right) +c_3}{m_0}+d_\mathscr G\le c_2.
$$
 	Without loss of generality, we take $c_2=c_4$. 	Therefore, $c_0$ is well-defined.  
 	\end{itemize}
 This completes the proof of the lemma. 
\end{proof}

Observe that if $(x,\xi)\in\mathscr K\in Y^*$ is a fixed point of $\mathcal Q\colon \mathscr K\times Y^*\to 2^{\mathscr K\times Y^*}$ defined by 
$$
\mathcal Q(z,\xi):=\left(\mathscr S(z,\xi),\mathscr G(\gamma z)\right),
$$
then it solve Problem~\ref{problems1} as well. Following this important property, we are going to show that $\mathcal Q$ has a fixed point in $\mathscr K\times Y^*$.  

Indeed, from the proof of Lemma~\ref{lemmas3}, we have the following lemma. 

\begin{lemma}\label{lemmas4}
Under the assumptions of  Lemma~\ref{lemmas2} hold, the inclusion is available 
$$
\mathcal Q(\mathscr D)\subset \mathscr D,
$$
where $\mathscr D=\mathcal D_1\times \mathcal D_2$ and $\mathcal D_1$ and $\mathcal D_2$ are defined by 
$$
\mathcal D_1:=\{x\in \mathscr K\,\mid\,\|x\|_X\le c_0\mbox{ and $\|\mathscr Lx\|_{X^*}\le c_5$}\},
$$
and $\mathcal D_2=B_{Y^*}(0,c_2)$, where $c_0$ and $c_2$ are given in the proof of Lemma~\ref{lemmas3}, and $c_5$ is defined by  
$$
c_5:=\frac{1}{d_0}\left[d_\mathscr F+\|\mathscr E\|_{X^*}+\|\gamma\|c_2+l_\Psi\right](c_0+d_0),
$$
here $d_\mathscr F:=\sup_{x\in \mathcal D_3}\|\mathscr F(x)\|_{X^*}$, $l_\Psi:=\sup_{z\in \mathcal D_3,x\in \mathcal D_3,y\in O_X(0,d_0)}b_\Psi(z,x,y)$ with $\mathcal D_3:=\{x\in X\,\mid\,\|x\|_X\le c_0\}$, and $d_0>0$ is such that $O_X(0,d_0)\subset \mathscr M(z)$ for all $z\in \mathscr K$.
\end{lemma}
\begin{proof}
From Lemma~\ref{lemmas3}, we can observe that for any $(z,\xi)\in \mathscr D$ it holds $\|\mathscr S(z,\xi)\|_X\le c_0$ and $\|\mathscr G(\gamma z)\|_{Y}\le c_2$. So, it remains us to verify that 
\begin{eqnarray}
	\|\mathscr L(\mathscr S(z,\xi))\|_{X^*}\le c_5. 
\end{eqnarray}

Because  $O_X(0,d_0)\subset \mathscr M(z)$ for all $z\in \mathscr K$. Then, we infer (more details, see the proof of Lemma~\ref{lemmas2})
\begin{align*}
	\|\mathscr Lx\|_{X^*}&=\sup_{y\in O_X(0,d_0)}\frac{1}{d_0}\langle\mathscr L(x),-y\rangle\\
	&\le\frac{1}{d_0}\left[d_\mathscr F+\|\mathscr E\|_{X^*}+\|\gamma^*\xi\|_{X^*}+l_\Psi\right](c_0 +d_0)\\
	&\le\frac{1}{d_0}\left[d_\mathscr F+\|\mathscr E\|_{X^*}+\|\gamma^*\xi\|_{X^*}+l_\Psi\right](c_0+d_0)\\
	&\le \frac{1}{d_0}\left[d_\mathscr F+\|\mathscr E\|_{X^*}+\|\gamma\|c_2+l_\Psi\right](c_0+d_0)\\
	&=c_5.
\end{align*}
The above estimates indicate that $\mathcal Q$ maps $\mathscr D$ into itself. 
\end{proof}

Based on Lemma~\ref{lemmas4}, to prove the emptiness of solution set of Problem~\ref{problems1} is sufficient to illustrate that multivalued mapping $\mathcal Q$ has a fixed point in $\mathscr D$. We are now in a position to provide the detailed proof of Theorem~\ref{theorems1}. 

\begin{proof}[Proof of Theorem~\ref{theorems1}.] The proof is divided into four steps.   

\noindent
{\bf Step 1.} {\it Problem~\ref{problems1} has at least one solution, when $\mathscr F$ satisfies strictly monotone.}

To this end, we shall employ Theorem~\ref{KlugeFPT} to conclude that $\mathcal Q\colon \mathscr D\to 2^{\mathscr D}$ has a fixed point in $\mathscr D$. In fact, we have 
\begin{itemize}
	\item[$\bullet$] From hypotheses H($\mathscr G$), we can observe that $\mathcal Q$ has nonempty, bounded, closed and convex values in $\mathcal W\times Y^*$. 
	\item[$\bullet$] Let $\{(z_n,\xi_n)\},\{x_n,\eta_n\}\subset \mathscr D$ be such that $x_n=\mathscr S(z_n,\xi_n)$, $\eta_n\in \mathscr  G(\gamma z_n)$, $(z_n,\xi_n)\wto (z,\xi)$ and $(x_n,\eta_n)\wto (x,\eta)$ in $\mathcal W\times Y^*$ for some $(z,\xi),(x,\eta)\in \mathscr D$ (because of the closedness and convexity of $\mathscr D$). 
	So, for every $n\in\mathbb N$, it yields
		\begin{eqnarray}\label{eqns3.10}
		\langle\mathscr L(x_n)+\mathscr F(x_n)-\mathscr E,y-x_n\rangle_X+\langle\xi_n,\gamma(y-x_n)\rangle_Y\ge \Psi(z_n,x_n)-\Psi(z_n,y) 
	\end{eqnarray} 	
	for all $y\in \mathscr M(z_n)\cap D(\mathscr L)$. 
		Let $y\in \mathscr M(z)$ be fixed. 
		Applying hypothesis H($\mathscr M$)(ii), it allows us to find a sequence $\{y_n\}\subset D(\mathscr L)$ such that $y_n\in \mathscr M(z_n)$ and $y_n\to y$ in $X$. Putting $y=y_n$ in (\ref{eqns3.10}), we can utilize the monotonicity of $\mathscr L$ and $\mathscr F$ to attain 
		\begin{align*}
			&\langle\mathscr L(y)+\mathscr F(y)-\mathscr E,x_n-y_n\rangle_X\\
				\le &\langle\mathscr L(y)+\mathscr F(y)-\mathscr E,x_n-y_n\rangle_X+\langle\mathscr L(x_n)+\mathscr F(x_n)-\mathscr E,y_n-x_n\rangle_X\\
				&+\langle\xi_n,\gamma(y_n-x_n)\rangle_Y-\Psi(z_n,x_n)+\Psi(z_n,y_n) \\
				=&\langle \mathscr L(x_n)+\mathscr F(x_n)-\mathscr L(y)-\mathscr F(y),y-x_n\rangle_X+\langle\xi_n,\gamma(y_n-x_n)\rangle_Y\\
				&+\langle \mathscr L(x_n)+\mathscr F(x_n)-\mathscr L(y)-\mathscr F(y),y_n-y\rangle_X-\Psi(z_n,x_n)+\Psi(z_n,y_n) \\
				\le&\langle \mathscr L(x_n)+\mathscr F(x_n)-\mathscr L(y)-\mathscr F(y),y_n-y\rangle_X+\langle\xi_n,\gamma(y_n-x_n)\rangle_Y\\
				&-\Psi(z_n,x_n)+\Psi(z_n,y_n).
			\end{align*}
			Passing to the upper limit for the inequality above, it yields 
			\begin{align*}
				&\langle\mathscr L(y)+\mathscr F(y)-\mathscr E,x-y\rangle_X\\
					=&\lim_{n\to\infty}\langle\mathscr L(y)+\mathscr F(y)-\mathscr E,x_n-y_n\rangle_X\\
				\le 	&\limsup_{n\to\infty}\bigg[\langle  \mathscr L(x_n)+\mathscr F(x_n)-\mathscr L(y)-\mathscr F(y),y_n-y\rangle_X+\langle\xi_n,\gamma(y_n-x_n)\rangle_Y\\
						&-\Psi(z_n,x_n)+\Psi(z_n,y_n)\bigg]\\
						=& \lim_{n\to\infty} \langle  \mathscr L(x_n)+\mathscr F(x_n)-\mathscr L(y)-\mathscr F(y),y_n-y\rangle_X+\lim_{n\to\infty}\langle\xi_n,\gamma(y_n-x_n)\rangle_Y\\
						&+\limsup_{n\to\infty}\left[-\Psi(z_n,x_n)+\Psi(z_n,y_n)\right]\\
						\le& \langle\xi,\gamma(y-x)\rangle_Y-\Psi(z,x)+\Psi(z,y),
				\end{align*}
				here it has used hypotheses H($\Psi$)(iii) and the compactness of $\gamma$ on $\mathcal W$. Then, it is available that 
				\begin{eqnarray*}
					\langle\mathscr L(y)+\mathscr F(y)-\mathscr E,y-x\rangle_X+\langle\xi,\gamma(y-x)\rangle_Y-\Psi(z,x)+\Psi(z,y)\ge 0
				\end{eqnarray*}
				 for all $y\in \mathscr M(z)$. Keeping in mind that $x\in \mathscr M(z)$ (owing to hypothesis H($\mathscr M$)(i)), from the hemicontinuity of $\mathscr F$ and $\mathscr L$ in $D(\mathscr L)\cap \mathscr M(z)$, and the convexity of $\Psi(z,\cdot)$, it is not hard to use Minty approach for getting that $x\in \mathscr M(z)$ solves the inequality 
				 \begin{eqnarray*}
				 	\langle\mathscr L(x)+\mathscr F(x)-\mathscr E,y-x\rangle_X+\langle\xi,\gamma(y-x)\rangle_Y-\Psi(z,x)+\Psi(z,y)\ge 0
				 \end{eqnarray*}
				for all $y\in \mathscr M(z)$. This indicates that $x=\mathscr S(z,\xi)$. 
				On the other hand, because $\mathscr G\colon Y\to 2^{Y^*}$ has a sequentially strongly-weakly closed  graph and $\gamma$ is compactness on $\mathcal W$, so, it is obvious that $\eta\in \mathscr G(\gamma z)$. This means that $(x,\eta)\in (\mathscr S(z,\xi),\mathscr G(\gamma z))=\mathcal Q(z,\xi)$, namely, the graph Gr($\mathcal Q$) of $\mathcal Q$ is weakly closed in $\mathcal W\times Y^*$. 
\end{itemize}

Under the analysis above, it allows us to invoke Kluge's fixed point theorem, Theorem~\ref{KlugeFPT}, for obtaining that $\mathscr Q$ has at least one fixed point $(x,\xi)\in\mathscr D$. Also, $x$ is a solution of Problem~\ref{problems1}. 

\smallskip
\noindent
{\bf Step 2.} {\it The solution set of Problem~\ref{problems1} is nonempty without the assumption that $\mathscr F$ is strictly monotone.}

In Step 1, we use the strict monotonicity of $\mathscr F$ to show the existence of solutions of Problem~\ref{problems1}. In fact, this assumption guarantees that $\mathscr S$ is singleton. Next, we will apply approximating method to remove this assumption.  
Let $\{\varepsilon_n\}\subset(0,+\infty)$ be such that $\varepsilon_n\to 0$ as $n\to\infty$. We consider the following perturbated inequality: find $x_n\in \mathscr M(x_n)\cap D(\mathscr L)$ and $\xi_n\in \mathscr G(\gamma x_n)$ such that 
\begin{eqnarray}\label{eqns3.13s}
	\langle\mathscr L(x_n)+\mathscr F(x_n)+\varepsilon_n \mathcal J(x_n)-\mathscr E,y-x_n\rangle_X+\langle\xi,\gamma(y-x_n)\rangle_Y\ge \Psi(x_n,x_n)-\Psi(x_n,y) 
\end{eqnarray} 	
for all $y\in \mathscr M(x_n)\cap D(\mathscr L)$, where $\mathcal J\colon X\to X^*$ is the normalized duality map on $X$. Let $x_n\in \mathscr M(x_n)\cap D(\mathscr L)$ be a solution of inequality (\ref{eqns3.13s}). We assert that $\{x_n\}$ is bounded in $\mathcal W$. Putting $y=0$ in (\ref{eqns3.13s}), it utilizes the positivity of $\langle\mathcal J(x_n),x_n\rangle$ and hypotheses H($\mathscr F$), H($\Psi$) and H($\mathscr G$) to having  
	\begin{align*}
	& \left(\|\mathscr E\|_{X^*}+\|\gamma\|(c_\mathscr G\|\gamma\|^{p-1}\|x_n\|_X^{p-1}+d_\mathscr G)\right)\|x_n\|_X\\
	\ge 	& \left(\|\mathscr E\|_{X^*}+\|\gamma^*\xi_n\|_{X^*}\right)\|x_n\|_X\\
	\ge &-\langle\xi,\gamma x_n\rangle_Y-\langle\mathscr E,x_n\rangle_X\\
	\ge &	\langle\mathscr L(x_n)+\mathscr F(x_n)+\varepsilon_n\mathcal J(x_n) ,x_n\rangle_X-\Psi(x_n,0)+\Psi(x_n,x_n)\\
	\ge &c_\mathscr F\|x\|_X^p-d_\mathscr F-c_\Psi\|x\|_X^\beta-d_\Psi-e_\Phi.
\end{align*}
Whereas, the smallness condition (\ref{smallnesscondition}) reveals that $\{x_n\}$ is bounded in $X$. Employing the same arguments as in the proof of Lemma~\ref{lemmas2}, it shows that $\{\mathscr Lx_n\}$ is bounded in $X^*$. The assertion has been proved. Therefore, it could say that 
\begin{eqnarray*}
	x_n\wto x\mbox{ in $X$ and $\mathscr Lx_n\wto \mathscr Lx$ in $X^*$}. 
\end{eqnarray*}
For every $y\in\mathscr M(x)$ fixed, it could find a sequence $\{y_n\}$ with $y_n\in \mathscr M(x_n)$ and $y_n\to y$ in $X$. As in (\ref{eqns3.13s}), we take $y=y_n$ to find  
\begin{align*}
	&\langle\mathscr L(y)+\mathscr F(y)+\varepsilon_n\mathcal J(y)-\mathscr E,x_n-y_n\rangle_X\\
	\le&\langle \mathscr L(x_n)+\mathscr F(x_n)+\varepsilon_n\mathcal J(x_n)-\mathscr L(y)-\mathscr F(y),y_n-y\rangle_X+\langle\xi_n,\gamma(y_n-x_n)\rangle_Y\\
	&-\Psi(x_n,x_n)+\Psi(x_n,y_n).
\end{align*}
Taking the upper limit as $n\to\infty$ for the inequality above and using Minty technique, it concludes that $x\in \mathscr M(x)\cap D(\mathscr L)$ is a solution of Problem~\ref{problems1}.

\smallskip
\noindent
{\bf Step 3.} {\it The solution set of Problem~\ref{problems1} is weakly compact in $\mathcal W$.} 

Let $\{x_n\}$ be a solution sequence of Problem~\ref{problems1}. Inserting $y=0$  
in (\ref{eqns3.1}) with $x=x_n$ and $\xi=\xi_n\in\mathscr G(\gamma x_n)$, we have 
	\begin{align*}
	& \left(\|\mathscr E\|_{X^*}+\|\gamma\|(c_\mathscr G\|\gamma\|^{p-1} \|x_n\|_X^{p-1}+d_\mathscr G)\right)\|x_n\|_X\\
	\ge 	& \left(\|\mathscr E\|_{X^*}+\|\gamma^*\xi_n\|_{X^*}\right)\|x_n\|_X\\
	\ge &-\langle\xi_n,\gamma x_n\rangle_Y\\
	\ge &	\langle\mathscr L(x_n)+\mathscr F(x_n) ,x_n\rangle_X-\Psi(x_n,0)+\Psi(x_n,x_n)\\
	\ge&c_\mathscr F\|x_n\|_X^p-d_\mathscr F-\Psi(x_n,0)+\Psi(x_n,x_n)\\
	\ge &c_\mathscr F\|x_n\|_X^p-d_\mathscr F-c_\Psi\|x_n\|_X^\beta-d_\Psi-e_\Psi,
\end{align*} 
that is, 
\begin{align*}
\left(\|\mathscr E\|_{X^*}+\|\gamma\|d_\mathscr G\right)\|x_n\|_X	+c_\Psi\|x_n\|_X^\beta+d_\mathscr F+d_\Psi+e_\Psi\ge &\left(c_\mathscr F-c_\mathscr G\|\gamma\|^p\right)\|x_n\|_X^p.
\end{align*}
This indicates that $\{x_n\}$ is bounded in $X$, due to $0< \beta<p$ and $c_\mathscr F-c_\mathscr G\|\gamma\|^p>0$. As before we did in Lemma~\ref{lemmas4}, it could find 
	\begin{align*}
	\langle\mathscr L(x_n),-z\rangle
	\le &\langle\mathscr F(x_n)-\mathscr E,z-x_n\rangle_X+\langle\xi_n,\gamma(z-x_n)\rangle_Y- b_\Psi(x_n,x_n,z)\|x_n-z\|_X^\eta
\end{align*} 	
for all $z\in O_X(0,d_0)$. It can verify that $\{\mathscr Lx_n\}$ is bounded in $X^*$. So, $\{x_n\}$ is bounded in $\mathcal W$. Without loss of generality, we may assume that
\begin{eqnarray}\label{eqns3.11}
	x_n\wto x\mbox{ in $\mathcal W$}. 
\end{eqnarray}
The boundedness of $\mathscr G$ permits us to suppose that $\xi_n\wto \xi$ in $Y^*$ for some $\xi\in  Y^*$. 
But, from the strong-weak closedness of $\mathscr G$ and compactness of $\gamma$ on $\mathcal W$, we infer $\xi\in \mathscr G(\gamma x)$. Remembering that $\mathcal Q$ is weakly closed. So, $\mathscr S$ is weakly closed in $\mathcal W\times Y^*$ as well. This implies that $x_n=\mathscr S(x_n,\xi_n)\wto \mathscr S(x,\xi)=x$ in $\mathcal W$ with $\xi\in \mathscr G(\gamma x)$.  
It shows that the solution set of Problem~\ref{problems1} is weakly compact in $\mathcal W$. 

\smallskip
\noindent
{\bf Step 4.} {\it When $\mathscr F$ satisfies $(S_+)$-property with respect to $\mathscr L$, then the solution set of Problem~\ref{problems1} is compact in $X$.}

From the proof of Step 3, we can see that $x_n\wto x$ in $\mathcal W$. Condition H($\mathscr M$)(ii) guarantees the existence of a sequence $\{z_n\}$ with $z_n\in \mathscr M(x_n)$ and $z_n\to x$ in $X$. Letting $y=z_n$ in (\ref{eqns3.1}) with  $x=x_n$ and $\xi=\xi_n\in\mathscr G(\gamma x_n)$, we deduce 
	\begin{align*}
	&\langle\mathscr L(x)-\mathscr E,x_n-z_n\rangle_X+\langle\mathscr F x_n,x_n-z_n\rangle\\
\le&\langle \mathscr L(x_n)-\mathscr L(x),x_n-x\rangle_X+\mathscr L(x_n)-\mathscr L(x),x-z_n\rangle_X+\langle\xi_n,\gamma(y_n-x_n)\rangle_Y\\
&
-\Psi(x_n,x_n)+\Psi(x_n,y_n)\\
\le&\mathscr L(x_n)-\mathscr L(x),x-z_n\rangle_X+\langle\xi_n,\gamma(y_n-x_n)\rangle_Y
-\Psi(x_n,x_n)+\Psi(x_n,y_n).
\end{align*}
Taking upper limit as $n\to\infty$, it derives 
\begin{eqnarray*}
	\limsup_{n\to\infty}\langle\mathscr F x_n,x_n-x\rangle\le 0.
\end{eqnarray*}
Then, $(S_+)$-property   of $\mathscr F$ leads to $x_n\to x$ in $X$.  
\end{proof}

In the end of this section, let us discuss several particular cases of Problem~\ref{problems1} and its corresponding existence and compactness results by using Theorem~\ref{theorems1}. 

Let $\mathscr J\colon Y\to \mathbb R$ be a locally Lipschitz function such that there are  constants $c_\mathscr J>0$, $d_\mathscr J\ge 0$ and $0<\theta\le p-1$ satisfying 
\begin{eqnarray}\label{eqns3.12}
	\|\eta\|_{Y*}\le c_\mathscr J\|y\|_Y^{\theta}+d_\mathscr J
\end{eqnarray}
for all $\eta\in\partial_C \mathscr J(y)$ and all $y\in Y$, where $\partial_C\mathscr J$ is the generalized Clarke subdifferential of $\mathscr J$. When $\mathscr G$ is formulated by $\mathscr G(y)=\partial_C\mathscr J(y)$ for all $y\in Y$, then Problem~\ref{problems1} becomes to the following one:
\begin{problem}\label{problems3}
	Given an element $\mathscr E\in X^*$, find $x\in \mathscr M(x)\cap D(\mathscr L)$ and $\xi\in \partial_C \mathscr J(\gamma x)$ such that 
	\begin{eqnarray}\label{eqns3.13}
		\langle\mathscr L(x)+\mathscr F(x)-\mathscr E,y-x\rangle_X+\langle\xi,\gamma(y-x)\rangle_Y\ge \Psi(x,x)-\Psi(x,y) 
	\end{eqnarray} 	
	for all $y\in \mathscr M(x)\cap D(\mathscr L)$.
\end{problem}
Then, employing Theorem~\ref{theorems1}, we have the following existence theorem for Problem~\ref{problems3}. 

\begin{theorem}\label{theorems2}
	Suppose that  H($\mathscr L$), H($\mathscr K$), H($\mathscr F$),  H($\mathscr E$), H($\gamma$), H($\Psi$) and H($\mathscr M$) are fulfilled. Also, we assume that $\mathscr J\colon Y\to \mathbb R$ is a locally Lipschitz function such that growth condition (\ref{eqns3.12}) is fulfilled,  and the inequality holds
	\begin{eqnarray*}
		c_\mathscr F>c_\mathscr J\|\gamma\|^p
	\end{eqnarray*}
	when $\theta=p-1$. Then, the solution set of Problem~\ref{problems3} is nonempty and weakly compact in $\mathcal W$. Moreover, if $\mathscr F$ satisfies $(S_+)$-property with respect to $\mathscr L$, then the solution set of Problem~\ref{problems3} is compact in $X$ as well.
\end{theorem} 
\begin{proof}
	It is sufficient to show that $\partial_C \mathscr J\colon Y\to 2^{Y^*}$ enjoys all properties of H($\mathscr G$). From Proposition~\ref{subdiff}, we can see that $\partial_C \mathscr J$ has a strongly-weakly closed graph with nonempty, closed, convex and bounded values. If $\theta=p-1$, then growth condition (\ref{eqns3.12}) is a direct consequence of (\ref{Ggrowthcondition}). When $\theta<p-1$, then it follows from Young inequality that 
	\begin{eqnarray*}
		\|\eta\|_{Y^*}\le c_\mathscr  J\|y\|_Y^\theta+d_\mathscr J\le \frac{c_\mathscr F}{2\|\gamma\|^p}\|y\|_Y^{p-1}+e_\mathscr J\mbox{ for all $\eta\in \partial_C\mathscr  J(y)$ and all $y\in Y$}
	\end{eqnarray*}
	for some $e_\mathscr J>0$ which is independent of $\eta$ and $y$. 
\end{proof}

It can observe that if $x\in \mathscr K$ is a solution of Problem~\ref{problems3}, then there exists $\xi\in \partial_C\mathscr J(\gamma x)$ such that (\ref{eqns3.13}) is available. Whereas, by the definition of Clarke subgradient (see Definition~\ref{SUB}), it derives 
	\begin{eqnarray*}
	\langle\mathscr L(x)+\mathscr F(x)-\mathscr E,y-x\rangle_X+\mathscr J^0(\gamma x;\gamma(y-x))\ge \Psi(x,x)-\Psi(x,y) 
\end{eqnarray*} 	
for all $y\in \mathscr M(x)\cap D(\mathscr L)$, namely,  
the solution set of Problem~\ref{problems3} is a subset of the following evolution quasi-variational-hemivariational inequality:
\begin{problem}\label{problems4}
	Given an element $\mathscr E\in X^*$, find $x\in \mathscr M(x)\cap D(\mathscr L)$   such that 
	\begin{eqnarray}\label{eqns3.14}
		\langle\mathscr L(x)+\mathscr F(x)-\mathscr E,y-x\rangle_X+\mathscr J^0(\gamma x;\gamma(y-x))\ge \Psi(x,x)-\Psi(x,y) 
	\end{eqnarray} 	
	for all $y\in \mathscr M(x)\cap D(\mathscr L)$.
\end{problem}
So, we have the following corollary. 
\begin{corollary}\label{corollarys1}
Under the assumptions of Theorem~\ref{theorems2}, the solution set of Problem~\ref{problems4} is nonempty and weakly compact in $\mathcal W$. Moreover, if $\mathscr F$ satisfies $(S_+)$-property with respect to $\mathscr L$, then the solution set of Problem~\ref{problems4} is compact in $X$.
\end{corollary} 

\begin{proof}
The existence part is a consequence of Theorem~\ref{theorems2}. The weak compactness of solution set of  Problem~\ref{problems4} in $\mathcal W$ could be proved by using the same arguments as in the proof of Theorem~\ref{theorems1} and the fact that for each $x\in \mathscr K$ there exist $\xi_x\in \partial_C \mathscr J(\gamma x)$ such that 
\begin{eqnarray*}
	\mathscr J^0(\gamma x;-\gamma(x))=\langle\xi,-\gamma x \rangle_Y.
\end{eqnarray*}
Likewise, it can also apply the similar way to show that the solution set of Problem~\ref{problems4} is compact in $X$ via employing the upper semicontinuity of $Y\times Y\ni (x,y)\mapsto\mathscr J^0(x;y)\in\mathbb R$, if  $\mathscr F$ satisfies $(S_+)$-property with respect to $\mathscr L$. 
\end{proof}

On the other hand, let us consider the special case of $\Psi$ that $\Psi$ is independent of the first variable. In such situation, Problem~\ref{problems1} becomes to the following one: 
\begin{problem}\label{problems5}
	Given an element $\mathscr E\in X^*$, find $x\in \mathscr M(x)\cap D(\mathscr L)$ and $\xi\in \mathscr G(\gamma x)$ such that 
	\begin{eqnarray}\label{eqns3.13bis}
		\langle\mathscr L(x)+\mathscr F(x)-\mathscr E,y-x\rangle_X+\langle\xi,\gamma(y-x)\rangle_Y\ge \Psi(x)-\Psi(y) 
	\end{eqnarray} 	
	for all $y\in \mathscr M(x)\cap D(\mathscr L)$.
\end{problem}
We also have the following existence result for Problem~\ref{problems5} in which the inequality $\Psi(z,x)\ge -c_\Psi\|x\|_X^\beta-d_\Psi$  for all $z,x\in X$, could be removed. 
\begin{theorem}\label{theorems4}
Under the assumptions of  H($\mathscr K$), H($\mathscr L$), H($\mathscr F$), H($\mathscr G$), H($\gamma$), H($\mathscr E$)  and H($\mathscr M$), if the inequality (\ref{smallnesscondition}) holds and  $\Psi\colon X\to \mathbb R$ is a convex and lower semicontinuous function,
then the solution set of Problem~\ref{problems5} is nonempty and weakly compact in $\mathcal W$.
Moreover, if $\mathscr F$ satisfies $(S_+)$-property with respect to $D(\mathscr L)$, then the solution set of Problem~\ref{problems5} is compact in $X$.
\end{theorem} 
\begin{proof}
	From the proof of Theorem~\ref{theorems1}, it can observe that the essential meaning of the inequality 
	\begin{eqnarray*}
		\mbox{$\Psi(z,x)\ge -c_\Psi\|x\|_X^\beta-d_\Psi$ for all $x\in X$}
		\end{eqnarray*}
	is to guarantee that the conclusion of Lemma~\ref{lemmas3} is valid. 
	In fact, when $\Psi$ is independent of its first variable (it could have infinite values), we have 
		\begin{align*}
		& \left(\|\mathscr E\|_{X^*}+\|\gamma\|\|\xi\|_{Y^*}\right)\|x\|_X\\
		\ge 	& \left(\|\mathscr E\|_{X^*}+\|\gamma^*\xi\|_{X^*}\right)\|x\|_X\\
		\ge &-\langle\xi,\gamma x\rangle_Y-\langle\mathscr E,x\rangle_X\\
		\ge &	\langle\mathscr L(x)+\mathscr F(x) ,x\rangle_X-\Psi(0)+\Psi(x)\\
		\ge&c_\mathscr F\|x\|_X^p-d_\mathscr F-\Psi(0)+\Psi(x)\\
		\ge &c_\mathscr F\|x\|_X^p-d_\mathscr F-\alpha_\Psi\|x\|_X-\beta_\Psi-\Psi(0),
	\end{align*}
	where $\alpha_\Psi,\beta_\Psi\ge 0$ are such that 
	\begin{eqnarray*}
		\Psi(x)\ge -\alpha_\Psi\|x\|_X-\beta_\Psi\mbox{ for all $x\in X$}. 
	\end{eqnarray*}
	Hence,
	\begin{align}\label{est1bis}
		& \left(\|\mathscr E\|_{X^*}+\|\gamma\|\|\xi\|_{Y^*}+\alpha_\Psi\right)\|x\|_X +\beta_\Psi+d_\mathscr F+\Psi(0)\ge c_\mathscr F\|x\|_X^{p}.
	\end{align}
	Therefore, we could apply the same arguments as in the proof of Lemma~\ref{lemmas3} to obtain the desired conclusion. 
	\end{proof}

However, when $\Psi$ has infinity values, then we have to strengthen the assumptions of $\mathscr M$  to guarantee the existence of solutions for Problem~\ref{problems5}. 
\begin{enumerate}
	\item[\textnormal{H($\mathscr M$)':}] $\mathscr M\colon \mathscr K\to 2^{\mathscr K}$ has nonempty, closed and convex values such that  $0\in \mbox{int}\left(\cap_{w\in \mathscr K}\mathscr M(w)\right)$ and 
	\begin{enumerate}
		\item[(i)] $x\in \mathscr M(y)\cap D(\mathscr L)$ holds, whenever $\{y_n\},\{x_n\}\subset \mathscr K$ fulfill $x_n\in \mathscr M(y_n)\cap D(\mathscr L)$, $y_n\wto y$   in $\mathcal W$ and $x_n\wto x$ in $\mathcal W$; 
		\item[(ii)] if $\{y_n\}\subset \mathscr K\cap D(\mathscr L)$ converges weakly to $y$ (i.e., $y_n\wto y$)  in $\mathcal W$, then for each $x\in \mathscr M(y)\cap D(\mathscr L)$ we could find two sequences $\{y_{n_k}\}\subset\{y_n\}$ and $\{x_k\}\subset X$ with $x_k\in \mathscr M(y_{n_k})\cap D(\mathscr L)$ such that $x_k\to x$ in $X$ and $\Psi(x_k)\to \Psi(x)$ as $k\to\infty$.
	\end{enumerate}
\item[\textnormal{H($\Psi$)':}] $\Psi\colon X\to\overline{\mathbb R}$ is a proper, convex and lower semicontinuous function such that $0\in \mbox{int}(D(\partial \Psi))$. \end{enumerate}

Applying the same arguments as in the proof of Theorem~\ref{theorems1}, we have the following theorem for Problem~\ref{problems5} with $\Psi\colon X\to\overline {\mathbb R}$. 

\begin{theorem}\label{theorems5}
Under the assumptions of  H($\mathscr K$), H($\mathscr L$), H($\mathscr F$), H($\mathscr G$), H($\Psi$)', H($\gamma$), H($\mathscr E$)  and H($\mathscr M$)', if the inequality (\ref{smallnesscondition}) holds, 
	then the solution set of Problem~\ref{problems5} is nonempty and weakly compact in $\mathcal W$. Moreover, if $\mathscr F$ satisfies $(S_+)$-property with respect to $\mathscr L$, then the solution set of Problem~\ref{problems5} is compact in $X$.
\end{theorem} 

\begin{remark}
	 The existence of solutions to Problem~\ref{problems5} was obtained  by Khan-Mig\'orski-Zeng~\cite{Khan-Migorski-Zeng-Optimization2024} under the strict assumptions, for example,  $\gamma \colon X\to Y$ is compact, $p=2$, $\Psi(0)=0$ and $\Psi(x)\ge -c_\Psi\|x\|_X$ for all $x\in X$. In fact, there are plenty of problems which can not satisfy these assumptions, for example, when $X=L^p(0,T;L^p(\Omega))$ and $Y=L^p(\Omega\times(0,T))$, then the operator $\gamma $ is not compact on $X$.  However, in this paper, we remove these strict requirements. This extends the scopes of applications of evolution multivalued quasi-variational inequalities.  
\end{remark}
 
 \section{Nonlienar and nonsmooth optimal control problems}\label{Section4}
 
 This section is devoted to develop a general framework for the study of a nonlinear and nonsmooth optimal control problem governed by evolution multivalued quasi-variational inequality, Problem~\ref{problems1}. More exactly, we are going to find the optimal control triple $\in \Pi$, $l\in \Theta$ and $\mathscr E\in  X^*$ such that a solution of evolution multivalued quasi-variational inequality associated with $(e,l,\mathscr E)$:
 	find $x\in \mathscr M(x)\cap D(\mathscr L)$ and $\xi\in \mathscr G(l,\gamma x)$ such that 
 	\begin{eqnarray}\label{eqns4.1}
 		\langle\mathscr L(x)+\mathscr F(e,x)-\mathscr E,y-x\rangle_X+\langle\xi,\gamma(y-x)\rangle_Y\ge \Psi(x,x)-\Psi(x,y) \mbox{ for all $y\in \mathscr M(x)\cap D(\mathscr L)$},
 	\end{eqnarray} 	
 approaches sufficiently to the known data which are measured in advance. Whereas, from the language of optimal control, we are interesting in the research of the following nonlinear and nonsmooth optimization problem:
    \begin{problem}\label{problems6}
    	Find optimal control pairs $e^*\in \Pi$, $l^*\in \Theta$ and $\mathscr E^*\in \Sigma$ such that the inequality  holds
    	\begin{eqnarray}\label{eqns4.0}
    		\mathcal C(e^*,l^*,\mathscr E^*)\le \mathcal C(e,l,\mathscr E)\mbox{ for all $(e,l,\mathscr E)\in \Pi\times\Theta\times\Sigma$},
    	\end{eqnarray}
    	where the cost functional $\mathcal C\colon \Pi\times\Theta\times \Sigma\to \mathbb R$ is defined by 
    	\begin{eqnarray*}
    		\mathcal C(e,l,\mathscr E):=\inf_{x\in S(e,l,\mathscr E)}h(x,e,l,\mathscr E),
    	\end{eqnarray*}
    	 $h\colon X\times \Pi\times \Theta\times\Sigma\to \mathbb R$ is a given function and $S(e,l,\mathscr E)$ is the solution set of evolution multivalued  quasi-variational inequality (\ref{eqns4.1}) corresponding to control variables (or unknown parameters) $(e,l,\mathscr E)$, $\Pi$, $\Theta$ and $\Sigma$ are the admissible sets for control variables (or unknown parameters) $e$, $l$ and $\mathscr E$, respectively.  
    	\end{problem}
    
    In order to establish the existence of optimal control variables (or parameters) $(e^*,l^*,\mathscr E^*)$, we impose the following assumptions on the data of Problem~\ref{problems6}:
    \begin{enumerate}
    \item[\textnormal{H($\Pi$):}] $\Pi_1$ is a Banach space such that $\Pi$ is weakly$^*$ closed in $\Pi_1$.
    \item[\textnormal{H($\Sigma$):}] $Z$ is a reflexive Banach space with compact embedding $Z\subset X^*$, and $\Sigma\subset Z$ is a weakly closed set. 
     \item[\textnormal{H($\Theta$):}] $\Theta_1$ is a Banach space such that $\Theta$ is weakly$^*$ closed in $\Theta_1$.
    \item[\textnormal{H($h$):}] $h\colon X\times\Theta\times \Pi\times \Sigma\to \mathbb R$ is a weakly lower semicontinuous function and satisfies the following inequality 
    \begin{eqnarray*}
    	h(x,e,l,\mathscr E)\ge r(e,l,\mathscr E)\mbox{ for all $(x,e,l,\mathscr E)\in X\times \Pi\times \Sigma$},
    \end{eqnarray*}
    where $r\colon \Pi_1\times\Theta_1 \times Z \to \mathbb R$ is bounded from below, and coercive on $\Pi_1\times \Theta_1\times Z$, namely, 
    \begin{eqnarray*}
    	r(e,l,\mathscr E)\to +\infty\mbox{ as $\|e\|_{\Pi_1}+\|l\|_{\Theta_1}+\|\mathscr E\|_{Z}\to \infty$}.
    \end{eqnarray*}
    \item[\textnormal{H($\mathscr F$)':}] $\mathscr F\colon \Pi\times X\to X^*$ is bounded and satisfies the following conditions:
    \begin{enumerate}
    	\item[{\rm(i)}] for each $e\in \Pi$, the function $X\ni x\mapsto\mathscr F(e,x)\in X^*$ is 
      monotone and hemicontinuous such that there exist two constants $c_\mathscr F>0$ and $d_\mathscr F\ge 0$ satisfying 
    \begin{eqnarray*}
    	\langle\mathscr F(e,x),x\rangle_X\ge c_\mathscr F\|x\|_X^p-d_\mathscr F\mbox{ for all $x\in X$ and all $e\in\Pi$},
    \end{eqnarray*}
    for some $1<p<+\infty$;
    \item[{\rm(ii)}] for each $y\in X$ fixed, $\Pi\ni e\mapsto F(e,y)\in X^*$ is continuous in the following sense that if $\{e_n\}\subset \Pi$ and $e\in \Pi$ are such that $e_n\to e$ weakly$^*$ in $\Pi_1$, then  $\mathscr F(e_n,y)\to \mathscr F(e,y)$ in $X^*$ as $n\to\infty$. 
        \end{enumerate}
        \item[\textnormal{H($\mathscr G$)':}]  $\mathscr G\colon\Theta\times Y\to 2^{Y^*}$ is a multivalued mapping such that for each $l\in \Theta$ the multivalued mapping $\mathscr G(l,\cdot)\colon Y\to 2^{Y*}$ satisfies hypothesis H($\mathscr G$) and 
        \begin{enumerate}
        	\item[]  $\Theta\times Y\ni (l,y)\mapsto \mathscr G(l,y)\subset Y^*$ is closed in the following sense that if $\{l_n\}\subset \Theta$, $\{y_n\}\subset Y$ and $(l,y)\in \Theta\times Y$ are such that $l_n\to l$ weakly$^*$ in $\Theta_1$ and $y_n\to y$ in $Y$ and $\xi_n\in \mathscr G(l_n,y_n)$ is such that $\xi_n\wto \xi$ in $Y^*$ for some $\xi\in Y^*$, then it holds $\xi\in \mathscr G(l,y)$. 
        \end{enumerate} 
    \end{enumerate}
    
    The following theorem delivers the nonemptiness and compactness of solution set to Problem~\ref{problems6}.
    
    \begin{theorem}\label{theorems6}
    	Under the assumptions of  H($\mathscr K$), H($\mathscr L$), H($\mathscr F$)', H($\mathscr G$)', H($\gamma$), H($\Psi$), H($\Theta$), H($\mathscr M$),  H($\Sigma$), H($\Pi$) and H($h$), if the inequality (\ref{smallnesscondition}) holds, 
    		then Problem~\ref{problems6} admits an optimal pairs $(e^*,l^*,\mathscr E^*)\in \Pi\times \Theta\times\Sigma$, and the solution set of of Problem~\ref{problems6} is weakly$^*$ compact in $\Pi\times\Theta\times \Sigma$. 
    \end{theorem} 
    \begin{proof}
First, we illustrate that the cost functional $\mathcal C\colon \Pi\times\Theta\times \Sigma\to \mathbb R$ is well-defined. For any $(e,l,\mathscr E)\in \Pi\times\Theta\times \Sigma$ fixed, it could use Theorem~\ref{theorems1} to see that $S(e,l,\mathscr E)$ is nonempty and weakly compactly in $\mathcal W$.  Because $h$ is bounded from the below, there exists a minimizing sequence $\{x_n\}\subset S(e,l,\mathscr E)$ to problem
\begin{eqnarray*}
	\inf_{x\in S(e,l,\mathscr E)}h(x,e,l,\mathscr E),
\end{eqnarray*}
namely,
\begin{eqnarray*}
		\inf_{x\in S(e,l,\mathscr E)}h(x,e,l,\mathscr E)=\lim_{n\to\infty}h(x_n,e,l,\mathscr E).
\end{eqnarray*}
Recall that $\{x_n\}\subset S(e,l,\mathscr E)$ and $S(e,l,\mathscr E)$ is weakly compact in $\mathcal W$, so, it could say that there is $x\in S(e,l,\mathscr E)$ such that 
\begin{eqnarray*}
	x_n\to x\mbox{ in $\mathcal W$}. 
\end{eqnarray*}
 Therefore, we can utilize the lower semicontinuity of $h(\cdot,e,l,\mathscr E)$ to get 
 \begin{align*}
 	\inf_{z\in S(e,l,\mathscr E)}h(z,e,l,\mathcal E)&=\lim_{n\to\infty}h(x_n,e,l,\mathscr E)\\
 	&\ge h(x,e,l,\mathscr E)\\
&\ge 		\inf_{z\in S(e,l,\mathscr E)}h(z,e,l,\mathscr E).
 \end{align*}
 This means that $\mathcal C\colon \Pi\times\Theta\times \Sigma\to \mathbb R$ is well-defined and for each $(e,l,\mathscr E)$ there exists $x\in S(e,l,\mathscr E)$ satisfying 
 \begin{eqnarray}\label{eqns4.3}
h(x,e,l,\mathscr E)=	\inf_{z\in S(e,l,\mathscr E)}h(z,e,l,\mathscr E).
 \end{eqnarray}
 
 Keeping in mind that $h$ is bounded from below, so, $\mathcal C$ is bounded from below as well. This allows us to pick up a minimizing sequence $\{(e_n,l_n,\mathscr E_n)\}\subset \Pi\times\Theta\times \Sigma$ for optimization problem (\ref{eqns4.0}) such that 
 \begin{eqnarray}\label{eqns4.5}
 \lim_{n\to\infty}	\mathcal C(e_n,l_n,\mathscr E_n)=\inf_{(e,l,\mathscr E)\in \Pi\times\Theta\times \Sigma}	\mathcal C(e,l,\mathscr E):=\rho>-\infty. 
 \end{eqnarray}
 Using hypothesis H($h$), we can observe that 
 \begin{eqnarray*}
 		\rho\ge \limsup_{n\to\infty}r(e_n,l_n,\mathscr E_n).
 \end{eqnarray*}
It shows that $\{e_n\}$ is bounded in $\Pi_1$, $\{l_n\}$ is bounded in $\Theta_1$, and $\{\mathscr E_n\}$ is bounded in $Z$. So, we infer that there are $(e^*,l^*,\mathscr E^*)\in \Pi\times\Theta \times\Sigma$ (due to the weakly closedness of $\Pi$, $\Theta$ and $\Sigma$) such that 
\begin{eqnarray}\label{eqns4.6}
	\left\{\begin{array}{lll}
			e_n\to e^*\mbox{ weakly$^*$ in $\Pi_1$},\\
			l_n\to l^*\mbox{ weakly$^*$ in $\Theta_1$},\\
			 \mathscr E_n\to\mathscr E^* \mbox{ weakly in $Z$ and strongly in $X^*$ (see hypothesis H($\Sigma$))}. 
		\end{array}\right.
\end{eqnarray}
Let sequence $\{x_n\}$ be such that $x_n\in S(e_n,l_n,\mathscr E_n)$ and 
\begin{eqnarray}\label{eqns4.7}
h(x_n,e_n,l_n,\mathscr E_n)=	\inf_{z\in S(e_n,l_n,\mathscr E_n)}h(z,e_n,l_n,\mathscr E_n).
\end{eqnarray}
We claim that $\{x_n\}$ is bounded in $\mathcal W$. The estimates 
	\begin{align*}
	& \left(\|\mathscr E_n\|_{X^*}+\|\gamma\|(c_\mathscr G\|\gamma\|^{p-1} \|x_n\|_X^{p-1}+d_\mathscr G)\right)\|x_n\|_X\\
	\ge&c_\mathscr F\|x_n\|_X^p-d_\mathscr F-\Psi(x_n,0)+\Psi(x_n,x_n)\\
	\ge &c_\mathscr F\|x_n\|_X^p-d_\mathscr F-c_\Psi\|x_n\|_X^\beta-d_\Psi,
\end{align*} 
reveal that $\{x_n\}$ is bounded in $X$ owing to inequality (\ref{smallnesscondition}). By hypotheses H($\mathscr M$), it may suppose that $B_X(0,d_0)\subset \mathscr M(x_n)$ for all $n\in\mathbb N$. Then, for every $z\in B_X(0,d_0)$, it is true
\begin{align*}
	\langle\mathscr L(x_n),-z\rangle
	\le &\langle\mathscr F(e_n,x_n)-\mathscr E_n,z-x_n\rangle_X+\langle\xi_n,\gamma(z-x_n)\rangle_Y- b_\Psi(x_n,x_n,z)\|x_n-z\|_X^\eta
\end{align*} 	
with  $\xi_n\in\mathscr G(l_n,\gamma x_n)$. The inequality above and the boundedness of of $\mathscr F$, $\mathscr G$, $b_\Psi$ and $\{x_n\}$ imply that 
$\{\mathscr Lx_n\}$ is bounded in $X^*$. Passing to a subsequence if necessary, we may say that 
\begin{eqnarray*}
	x_n\wto x\mbox{ in $X$ and $\mathscr Lx_n\wto \mathscr Lx$ in $X^*$}
\end{eqnarray*}
with $x\in \mathscr M(x)\cap D(\mathscr L)$ (see hypothesis H($\mathscr M$)(i)).  For every $y\in \mathscr M(x)$, we can use assumption H($\mathscr M$)(ii) to find a sequence $\{y_n\}$ satisfying 
\begin{eqnarray*}
	y_n\in\mathscr M(x_n)\mbox{ and $y_n\to y$ in $X$}.
\end{eqnarray*}
For any $n\in\mathbb N$, we have 
\begin{align*}
	\langle\mathscr L(x_n)+\mathscr F(e_n,x_n)-\mathscr E_n,y_n-x_n\rangle_X+\langle\xi_n,\gamma(y_n-x_n)\rangle_Y\ge \Psi(x_n,x_n)-\Psi(x_n,y_n).
\end{align*}
Condition H($\mathscr G$) results in the boundedness of $\{\xi_n\}$ in $Y^*$. But, the closedness of $\mathscr G$ (see hypothesis H($\mathscr G$)') and compactness of $\gamma$ in $\mathcal W$ guarantee that $\xi_n\wto \xi$ in $Y^*$ for some $\xi\in \mathscr G(l,\gamma x)$. 
Moreover, a simple calculation implies 
\begin{align*}
	&\langle\mathscr L(y)+\mathscr F(e_n,y),x_n-y_n\rangle_X\\
	\le&\langle\mathscr L(x_n)-\mathscr L(y)+\mathscr F(e_n,x_n)-\mathscr F(e_n,y),y-x_n\rangle_X+\langle-\mathscr E_n,y_n-x_n\rangle_X\\
	&+\langle\mathscr L(x_n)-\mathscr L(y)+\mathscr F(e_n,x_n)-\mathscr F(e_n,y),y_n-y\rangle_X+\langle\xi_n,\gamma(y_n-x_n)\rangle_Y\\
	&+\Psi(x_n,y_n)-\Psi(x_n,x_n)\\
		\le&\langle\mathscr L(x_n)-\mathscr L(y)+\mathscr F(e_n,x_n)-\mathscr F(e_n,y),y_n-y\rangle_X+\langle-\mathscr E_n,y_n-x_n\rangle_X\\
	&+\langle\xi_n,\gamma(y_n-x_n)\rangle_Y+\Psi(x_n,y_n)-\Psi(x_n,x_n).
	\end{align*}
	Taking the upper limit as $n\to\infty$ for the above estimate, it yields 
	\begin{align*}
			&\langle\mathscr L(y)+\mathscr F(e^*,y),x-y\rangle_X\\
		=&\lim_{n\to\infty}\langle\mathscr L(y)+\mathscr F(e_n,y),x_n-y_n\rangle_X \\
	\le&\lim_{n\to\infty}\langle\mathscr L(x_n)-\mathscr L(y)+\mathscr F(e_n,x_n)-\mathscr F(e_n,y),y_n-y\rangle_X+\lim_{n\to\infty}\langle-\mathscr E_n,y_n-x_n\rangle_X\\
	&+\lim_{n\to\infty}\langle\xi_n,\gamma(y_n-x_n)\rangle_Y+\limsup_{n\to\infty}\left[\Psi(x_n,y_n)-\Psi(x_n,x_n)\right]\\
	\le &\langle-\mathscr E^*,y-x\rangle_X+\langle\xi,\gamma(y-x)\rangle_Y+\Psi(x,y)-\Psi(x,x),
\end{align*}
where we have used the compactness of $Z$ to $X^*$. 
Therefore, it follows from the arbitrariness of $y\in \mathscr M(x)$ and the fact $\xi\in\mathscr G(l,\gamma x)$ that $x\in S(e,l,\mathscr E)$. 

Taking into account of (\ref{eqns4.5})--(\ref{eqns4.7}), one derives 
\begin{align*}
	\inf_{(e,l,\mathscr E)\in \Pi\times\Theta\times \Sigma}	\mathcal C(e,l,\mathscr E)=& \lim_{n\to\infty}	\mathcal C(e_n,l_n,\mathscr E_n)\\
	=&\lim_{n\to\infty}h(x_n,e_n,l_n,\mathscr E_n)\\
	\ge &h(x,e^*,l^*,\mathscr E^*)\\
	\ge & \mathcal C(e^*,l^*,\mathscr E^*)\\
	\ge &	\inf_{(e,l,\mathscr E)\in \Pi\times\Theta\times \Sigma}	\mathcal C(e,l,\mathscr E).
\end{align*}
This indicates that $(e^*,l^*,\mathscr E^*)\in \Pi\times\Theta\times \Sigma$ is an optimal control pair of Problem~\ref{problems6}. 

Furthermore, we shall verify that the solution set of Problem~\ref{problems6} is weakly compact in $\Pi\times\Theta\times \Sigma$. Let $\{(e_n,l_n,\mathscr E_n)\}\subset \Pi\times\Theta\times \Sigma$ be a solution sequence of Problem~\ref{problems6}. But, (\ref{eqns4.5}) says that $\{(e_n,l_n,\mathscr E_n)\}\subset \Pi\times\Theta\times \Sigma$ is bounded in $\Pi_1\times \Theta_1\times Z$. So, (\ref{eqns4.6}) holds. Employing the same arguments as in the proof of the first part, it could show that $(e^*,l^*,\mathscr E^*)\in \Pi\times \Theta\times \Sigma$ is a solution of Problem~\ref{problems6}, that is, the solution set of Problem~\ref{problems6} is weakly compact in $\Pi\times \Theta\times \Sigma$. 
    	\end{proof}
    	\begin{remark}
    		When $\mathscr F(e,\cdot)$ satisfies ($S_+$)-property with respect to $\mathscr L$, then the weakly lower semicontinuity of $h$ could be relaxed to the following one:
    		 \begin{itemize}
    		 	\item[]  $h\colon X\times\Theta\times \Pi\times \Sigma\to \mathbb R$ is  lower semicontinuous in the following sense, if $x_n\to x$ in $X$, $(e_n,l_n,\mathscr E_n)\to (e,l,\mathscr E)$ weakly$^*$ in $\Pi_1\times \Theta_1\times Z$ then the inequality is available
    		 	\begin{eqnarray*}
    		 		\liminf_{n\to\infty}h(x_m,e_n,l_n,\mathscr E_n)\ge h(x,e,l,\mathscr E).
    		 	\end{eqnarray*}
    		 \end{itemize}
    		 
    		On the other hand, Problem~\ref{problems6} can be a powerful and useful model to study various optimization or control problems driven by  evolution multivalued quasi-variational inequalities. For example, 
    		\begin{itemize}
    			\item[$\bullet$] when the control variables $(e,l,\mathscr E)$ are considered in the partial differential equations descried in the domain and certain part of boundary, then Problem~\ref{problems6} is a simultaneous distributed-boundary optimal control
    			problems driven by evolution multivalued quasi-variational inequalities.
    			\item[$\bullet$] if $(e,l,\mathscr E)$ are the unknown parameters for evolution multivalued quasi-variational inequality (\ref{eqns4.1}), then  Problem~\ref{problems6}  can be seen as an optimal identification model for the inverse parameters problem of evolution multivalued quasi-variational inequality (\ref{eqns4.1}).
    		\end{itemize}
    	\end{remark}
 	
 \section{Conclusions}\label{Section5}
 
We have studied a new class of evolution multivalued quasi-variational inequalities involving a nonlinear bifunction which contain several evolution quasi-variational inequalities as particular cases.  Under quite mild hypotheses, we employed an existence result for variational inequalities with a proper convex functional and a coercive maximal monotone operator, multivalued analysis and Kluge's fixed point theorem of multivalued version to prove the existence of solutions and compactness of solution set of the evolution multivalued quasi-variational inequalities under consideration. On the other side, we established a novel framework to solve a nonlinear and nonsmooth optimal control problem governed by evolution multivalued quasi-variational inequality, Problem~\ref{problems1}. Such nonlinear and nonsmooth optimal control problem could be applied to study simultaneous distributed-boundary optimal control
problems driven by evolution multivalued quasi-variational inequalities, optimal parameters identification for evolution multivalued quasi-variational inequalities, and so forth. Finally, we have to mention that the theoretical results established in this paper could be applied to research various  parabolic differential inclusions with nonlinear partial differential operators, semipermeability problems with mixed boundary conditions, and non-stationary Non-Newton fluid problems with multivalued and nonmonotone friction law, and so on (more details, one could refer our second paper~\cite{Zeng-Radulescu}).

\section*{Declarations}
\subsection*{Acknowledgments}

This project has received funding from the  Natural Science Foundation of Guang\-xi Grant Nos. 2021GXNSFFA196004 and GKAD23026237, the NNSF of China Grant No. 12371312,  and the European Union's Horizon 2020 Research and Innovation Programme under the Marie Sklodow\-ska-Curie grant agreement No. 823731 CONMECH. The first author is also supported by the project cooperation between Guangxi Normal University and Yulin Normal University. The research of V.D.~R\u adulescu was supported by the grant ``Nonlinear Differential Systems in Applied Sciences" of
the Romanian Ministry of Research, Innovation and Digitization, within PNRR-III-C9-2022-I8/22.

\subsection*{Data availability statement}
Data sharing not applicable to this article as no data sets were generated or analysed during the current study.

\subsection*{Ethical Approval}
Not applicable.

\subsection*{Competing interests} There is no conflict of interests.

\subsection*{Authors' contributions} The authors contributed equally to this paper.

\end{document}